\documentclass[11pt]{amsart}

\usepackage{amsmath}
\usepackage{amssymb}
\usepackage{amscd}
\usepackage{amsthm}
\usepackage{hyperref}
\usepackage{colortbl}

\usepackage{mathtools}
\usepackage{extarrows}

\usepackage[all]{xy}

\usepackage{tikz}
\usepackage{pgfplots}

\newtheorem{thm}{Theorem}[section]
\newtheorem{lem}[thm]{Lemma}

\newtheorem{prop}[thm]{Proposition}
\newtheorem{question}[thm]{Question}
\newtheorem{claim}[thm]{Claim}

\theoremstyle{remark}

\theoremstyle{definition}
\newtheorem{defn}[thm]{Definition}

\numberwithin{equation}{section}

\allowdisplaybreaks[4]

\begin{document}

\vfuzz0.5pc
\hfuzz0.5pc 

\newcommand{\claimref}[1]{Claim \ref{#1}}
\newcommand{\thmref}[1]{Theorem \ref{#1}}
\newcommand{\propref}[1]{Proposition \ref{#1}}
\newcommand{\lemref}[1]{Lemma \ref{#1}}
\newcommand{\coref}[1]{Corollary \ref{#1}}
\newcommand{\remref}[1]{Remark \ref{#1}}
\newcommand{\conjref}[1]{Conjecture \ref{#1}}
\newcommand{\questionref}[1]{Question \ref{#1}}
\newcommand{\defnref}[1]{Definition \ref{#1}}
\newcommand{\secref}[1]{\S \ref{#1}}
\newcommand{\ssecref}[1]{\ref{#1}}
\newcommand{\sssecref}[1]{\ref{#1}}

\newcommand{\RED}{{\mathrm{red}}}
\newcommand{\tors}{{\mathrm{tors}}}
\newcommand{\eq}{\Leftrightarrow}

\newcommand{\mapright}[1]{\smash{\mathop{\longrightarrow}\limits^{#1}}}
\newcommand{\mapleft}[1]{\smash{\mathop{\longleftarrow}\limits^{#1}}}
\newcommand{\mapdown}[1]{\Big\downarrow\rlap{$\vcenter{\hbox{$\scriptstyle#1$}}$}}
\newcommand{\smapdown}[1]{\downarrow\rlap{$\vcenter{\hbox{$\scriptstyle#1$}}$}}

\newcommand{\A}{{\mathbb A}}
\newcommand{\I}{{\mathcal I}}
\newcommand{\J}{{\mathcal J}}
\newcommand{\CO}{{\mathcal O}}
\newcommand{\C}{{\mathcal C}}
\newcommand{\BC}{{\mathbb C}}
\newcommand{\BQ}{{\mathbb Q}}
\newcommand{\m}{{\mathcal M}}
\newcommand{\h}{{\mathcal H}}
\newcommand{\Z}{{\mathcal Z}}
\newcommand{\BZ}{{\mathbb Z}}
\newcommand{\W}{{\mathcal W}}
\newcommand{\Y}{{\mathcal Y}}
\newcommand{\T}{{\mathcal T}}
\newcommand{\BP}{{\mathbb P}}
\newcommand{\CP}{{\mathcal P}}
\newcommand{\G}{{\mathbb G}}
\newcommand{\BR}{{\mathbb R}}
\newcommand{\D}{{\mathcal D}}
\newcommand{\DD}{{\mathcal D}}
\newcommand{\LL}{{\mathcal L}}
\newcommand{\f}{{\mathcal F}}
\newcommand{\E}{{\mathcal E}}
\newcommand{\BN}{{\mathbb N}}
\newcommand{\N}{{\mathcal N}}
\newcommand{\K}{{\mathcal K}}
\newcommand{\R} {{\mathbb R}}
\newcommand{\PP}{{\mathbb P}}
\newcommand{\Pp}{{\mathbb P}}
\newcommand{\BF}{{\mathbb F}}
\newcommand{\QQ}{{\mathcal Q}}
\newcommand{\closure}[1]{\overline{#1}}
\newcommand{\EQ}{\Leftrightarrow}
\newcommand{\imply}{\Rightarrow}
\newcommand{\isom}{\cong}
\newcommand{\embed}{\hookrightarrow}
\newcommand{\tensor}{\mathop{\otimes}}
\newcommand{\wt}[1]{{\widetilde{#1}}}
\newcommand{\ol}{\overline}
\newcommand{\ul}{\underline}

\newcommand{\bs}{{\backslash}}
\newcommand{\CS}{{\mathcal S}}
\newcommand{\CA}{{\mathcal A}}
\newcommand{\Q}{{\mathbb Q}}
\newcommand{\F}{{\mathcal F}}
\newcommand{\sing}{{\text{sing}}}
\newcommand{\U} {{\mathcal U}}
\newcommand{\B}{{\mathcal B}}
\newcommand{\X}{{\mathcal X}}

\newcommand{\ECS}[1]{E_{#1}(X)}
\newcommand{\CV}[2]{{\mathcal C}_{#1,#2}(X)}

\newcommand{\rank}{\mathop{\mathrm{rank}}\nolimits}
\newcommand{\codim}{\mathop{\mathrm{codim}}\nolimits}
\newcommand{\Ord}{\mathop{\mathrm{Ord}}\nolimits}
\newcommand{\Var}{\mathop{\mathrm{Var}}\nolimits}
\newcommand{\Ext}{\mathop{\mathrm{Ext}}\nolimits}
\newcommand{\EXT}{\mathop{{\mathcal E}\mathrm{xt}}\nolimits}
\newcommand{\Pic}{\mathop{\mathrm{Pic}}\nolimits}
\newcommand{\Spec}{\mathop{\mathrm{Spec}}\nolimits}
\newcommand{\Jac}{\mathop{\mathrm{Jac}}\nolimits}
\newcommand{\Div}{\mathop{\mathrm{Div}}\nolimits}
\newcommand{\sgn}{\mathop{\mathrm{sgn}}\nolimits}
\newcommand{\supp}{\mathop{\mathrm{supp}}\nolimits}
\newcommand{\Hom}{\mathop{\mathrm{Hom}}\nolimits}
\newcommand{\Sym}{\mathop{\mathrm{Sym}}\nolimits}
\newcommand{\nilrad}{\mathop{\mathrm{nilrad}}\nolimits}
\newcommand{\Ann}{\mathop{\mathrm{Ann}}\nolimits}
\newcommand{\Proj}{\mathop{\mathrm{Proj}}\nolimits}
\newcommand{\mult}{\mathop{\mathrm{mult}}\nolimits}
\newcommand{\Bs}{\mathop{\mathrm{Bs}}\nolimits}
\newcommand{\Span}{\mathop{\mathrm{Span}}\nolimits}
\newcommand{\IM}{\mathop{\mathrm{Im}}\nolimits}
\newcommand{\Hol}{\mathop{\mathrm{Hol}}\nolimits}
\newcommand{\End}{\mathop{\mathrm{End}}\nolimits}
\newcommand{\CH}{\mathop{\mathrm{CH}}\nolimits}
\newcommand{\Exec}{\mathop{\mathrm{Exec}}\nolimits}
\newcommand{\SPAN}{\mathop{\mathrm{span}}\nolimits}
\newcommand{\birat}{\mathop{\mathrm{birat}}\nolimits}
\newcommand{\cl}{\mathop{\mathrm{cl}}\nolimits}
\newcommand{\rat}{\mathop{\mathrm{rat}}\nolimits}
\newcommand{\Bir}{\mathop{\mathrm{Bir}}\nolimits}
\newcommand{\Rat}{\mathop{\mathrm{Rat}}\nolimits}
\newcommand{\aut}{\mathop{\mathrm{aut}}\nolimits}
\newcommand{\Aut}{\mathop{\mathrm{Aut}}\nolimits}
\newcommand{\eff}{\mathop{\mathrm{eff}}\nolimits}
\newcommand{\nef}{\mathop{\mathrm{nef}}\nolimits}
\newcommand{\amp}{\mathop{\mathrm{amp}}\nolimits}
\newcommand{\DIV}{\mathop{\mathrm{Div}}\nolimits}
\newcommand{\Bl}{\mathop{\mathrm{Bl}}\nolimits}
\newcommand{\Cox}{\mathop{\mathrm{Cox}}\nolimits}
\newcommand{\NE}{\mathop{\mathrm{NE}}\nolimits}
\newcommand{\NM}{\mathop{\mathrm{NM}}\nolimits}
\newcommand{\Gal}{\mathop{\mathrm{Gal}}\nolimits}
\newcommand{\coker}{\mathop{\mathrm{coker}}\nolimits}
\newcommand{\ch}{\mathop{\mathrm{ch}}\nolimits}

\title{$\A^1$ Curves on Log K3 Surfaces}

\author[X.~Chen]{Xi Chen${}^{\dagger}$}
\address{632 Central Academic Building\\
University of Alberta\\
Edmonton, Alberta T6G 2G1, Canada}
\email{xichen@math.ualberta.ca}

\author[Y.~Zhu]{Yi Zhu}
\address{Pure Mathematics\\
Univeristy of Waterloo\\
Waterloo, Ontario N2L3G1, Canada}
\email{yi.zhu@uwaterloo.ca}

\date{February 5, 2016}

\thanks{${}^{\dagger}$ Research partially supported by NSERC 262265.}
\begin{abstract}

In this paper, we study $\A^1$ curves on log K3 surfaces. We classify all genuine log K3 surfaces of type II which admits countably infinite $\A^1$ curves.
\end{abstract}

\maketitle

\section{Introduction}\label{LOGK3RATSECINTRO}

From the point of view of birational geometry, $\A^1$ curves play the roles for log varieties as rational curves do for projective varieties. However, much less is known in the log world, even in two dimensional case. $\A^1$ curves on log varieties with negative log Kodaira dimension are studied in \cite{Miyanishi-T2,K-M,CZ,A1,Z3}.

Inspired by the recent progress on the existence of countably many rational curves on a projective K3 surface (\cite{BHT}, \cite{C-L} and \cite{L-L}), we propose the following question studying $\A^1$ curves on log K3 surfaces classified by S. Iitaka \cite{I} and D. Q. Zhang 
\cite{Z}.

\begin{question}\label{LOGK3RATMAINQUESTION}
For which log K3 surfaces $(X,D)$, are there infinitely many $\A^1$ curves on $X\backslash D$?
\end{question}

A {\em log K3} surface, in the sense of Iitaka, is a log smooth projective pair $(X, D)$ satisfying
$h^0(K_X+D) = 1$
and $\kappa(X,D) = q(X,D) = 0$. According to Iitaka's classification, there are two types of log K3's:
\begin{description}
\item[Type I] $X$ is birational to a projective K3 surface;
\item[Type II] $X$ is a smooth projective rational surface.
\end{description}

In this paper, we are mainly interested in a special class of log K3 surfaces:

\begin{defn}
A {\em genuine log K3} surface is a log smooth projective surface pair $(X,D)$ such that 
\begin{enumerate}
\item $K_X+D = 0$ in $\Pic(X)$;
\item $q(X,D)=h^0(\Omega^1_X(\log D)) = 0$. 
\end{enumerate}
\end{defn}

In Iitaka's classification, genuine log K3 surfaces serve as the building blocks of log K3 surfaces. Of course, a genuine log K3 surface of type I is simply a projective K3 surface without boundary.
It has been proved by J. Li and C. Liedtke that there are infinitely many rational curves on almost every projective K3 surface $X$ (provided that $\rank_\BZ\Pic(X)$ is odd or $\rank_\BZ\Pic(X) \ge 5$
or $X$ has an elliptic fiberation) \cite{L-L}. So we have a nearly complete answer to Question \ref{LOGK3RATMAINQUESTION} for genuine log K3's of type I. In this paper, we study this question for genuine log K3's of type II.

Since the existence of $\A^1$ curves is essentially a property of the open part $X\backslash D$ of a log variety $(X,D)$, we consider $(X_1, D_1)$ and $(X_2, D_2)$ to be \emph{log isomorphic} 
if there exists a birational map $f: X_1\dashrightarrow X_2$
inducing an isomorphism $f: X_1 \backslash D_1 \cong X_2\backslash D_2$ and we call such $f$ a
{\em log isomorphism}. 
For genuine log K3's of type II, we have the following classification under log isomorphisms. 

\begin{thm}\label{LOGK3RATTHMLOGISOMK3}
Every genuine log K3 surface $(X,D)$ of type $\mathrm{II}$ is log isomorphic to one of the following genuine log K3 surfaces
$(\widehat{X}, \widehat{D})$:
\begin{enumerate}
\item[C0.] $\widehat{D}$ is a smooth elliptic curve;
\item[C1.] $\widehat{D}$ is a nodal rational curve;
\item[C2.] $\widehat{D} = \widehat{D}_1 + \widehat{D}_2 + ... + \widehat{D}_n$ is a circular boundary (see below)
satisfying $\widehat{D}_i^2 \le -2$ for $i\ne 1$ and $\widehat{D}_1^2 \ne 0,-1$;
\item[C3.] $\widehat{D} = \widehat{D}_1 + \widehat{D}_2$ is a circular
boundary satisfying $\widehat{D}_1^2 \ne -1$ and $\widehat{D}_2^2 = 0$;
\item[C4.] $\widehat{D} = \widehat{D}_1 + \widehat{D}_2$ is a circular boundary satisfying $\widehat{D}_1^2 > 0$ and $\widehat{D}_2^2 > 0$.
\end{enumerate}
\end{thm}

We have a complete answer to Question \ref{LOGK3RATMAINQUESTION} for genuine log K3 surfaces $(X,D)$ of Iitaka type $\text{II}$ by our main theorem:

\begin{thm}[$\A^1$ curves on genuine log K3's of Iitaka type $\text{II}$]\label{LOGK3RATTHM001}
Let $(X, D)$ be a genuine log K3 surface of type $\text{II}$. Then there are countably many $\A^1$ curves in $X\backslash D$ if and only if $(X,D)$ is log isomorphic to one of C0-C3 in Theorem \ref{LOGK3RATTHMLOGISOMK3}.
\end{thm}

It is relatively easier to prove the existence of infinitely many $\A^1$ curves on $(X,D)$ of type C0, C1 and C3 compared with C2.
For $(X,D)$ of type C2, we can contract $D_2 + D_3 + ... + D_n$  to obtain a {\em log del Pezzo} surface $\overline{X}$, i.e., a projective surface with at worst log terminal singularities and $-K_{\overline{X}}$ ample. Here is where the celebrated theorem of Keel-McKernan comes in: $\overline{X}$ is rationally connected \cite{K-M}. We will use this to show that there are infinitely many $\A^1$ curves in $\overline{X}\backslash \overline{D}$.

As suggested to us by David McKinnon, our construction of $\A^1$ curves on log K3 surfaces over number fields actually produce an infinite sequence of $\A^1$ curves defined over number fields of increasing degrees over $\BQ$.

\begin{thm}\label{LOGK3RATTHMMCKINNON}
For a genuine log K3 surface $(X,D)$ over $\overline{\BQ}$ where $D$ is either a smooth elliptic curve or a rational curve with one node,
there does not exist a number field $k\subset \overline{\BQ}$ such that
every $\A^1$ curve in $X\backslash D$ is defined over $k$. 
\end{thm}

Note that for a log K3 surface $(X,D)$ over $\overline{\BQ}$,
every $\A^1$ curve in $X_\BC \backslash D_\BC$ 
is automatically defined over $\overline{\BQ}$ due to rigidity
(see Lemma \ref{LOGK3RATLEMRIGIDITY}).

The similar statement for rational curves on K3 surfaces over number fields is expected but not known, to the best of our knowledge.
Although J. Li and C. Liedtke proved that almost all K3 surfaces over number fields have infinitely many rational curves, it is not clear that the rational curves they produced lie over an ascending chain of number fields.

The paper is organized as follows. Theorem \ref{LOGK3RATTHMLOGISOMK3} is proved in \S\ref{LOGK3RATSEC003}. In \S\ref{LOGK3RATSEC002}, we deal with genuine log K3 surfaces $(X,D)$ of type C0 and C1 and prove the existence of infinitely many $\A^1$ curves and Theorem \ref{LOGK3RATTHMMCKINNON} for such $(X,D)$. The rest of our main theorem \ref{LOGK3RATTHM001} is then proved in \S\ref{LOGK3RATSEC004}.
In \S\ref{LOGK3RATSECIITAKA}, we put our results under the framework of Iitaka's classification of log K3 surfaces and give examples of genuine log K3 surfaces that do not have infinitely many $\A^1$ curves.

\subsection{Remarks on Question \ref{LOGK3RATMAINQUESTION}}

Question \ref{LOGK3RATMAINQUESTION} is very difficult in general. For example, we do not know the case when $(X,D)$ is obtained from 
the blowup $\pi: X\to S$ of a K3 surface $S$ at finitely many points $\Sigma$ with $D$ the exceptional divisor of $\pi$; finding $\A^1$ curves in $X\backslash D$ amounts to finding rational curves on $S$ missing all but one point in $\Sigma$, which turns out to be a surprisingly difficult problem. An affirmative answer would generalize the theorem of Li-Lietke \cite{L-L}. On the other hand, there are log K3 surfaces with no log rational curves at all. For example, let $X$ be a Kummer K3 and let $D$ be the disjoint union of 16 $(-2)$-curves. Then $(X,D)$ is a log K3 with no $\A^1$ curves because $X\backslash D$ has an \'etale cover by an abelian surface deleting 16 points. Also there are many examples of genuine log K3 surfaces of type II without infinitely many $\A^1$ curves (see \S\ref{LOGK3RATSECIITAKA}). This suggests that the condition on the vanishing of log irregularity is too weak to ensure the existence of $\A^1$ curves.

\subsection*{Convention and terminology}

We work exclusively over algebraically closed fields of characteristic $0$.
Throughout the paper, ``countable'' means ``countably infinite''.

A log pair $(X,D)$ means a variety $X$ with a reduced Weil divisor $D$. Let $U$ be its interior $X-D$. We say that $(X,D)$ is \emph{log smooth} if $X$ is smooth and $D$ is a normal crossing (nc) divisor. A log pair is projective if the ambient variety is projective.

For a log smooth pair $(X,D)$, we use $\kappa(X,D)$ to denote the logarithmic Kodaira dimension and $q(X,D)$ to denote the logarithmic irregularity, i.e., $q(X,D)=h^0(X,\Omega_X^1(\log D))$. They only depend on the interior of the pair.

An $\A^1$ (or log rational) curve $C^\circ$ in $X\backslash D$ is a quasi-projective curve whose normalization is $\A^1$. Alternatively, the closure $C$ of $C^\circ$
in $X$ is a rational curve satisfying that $\nu^{-1}(D)$ consists of at most one point for the normalization $\nu: C^\nu\to X$ of $C$.

It is easy to see that a genuine log K3 surface $(X,D)$ of type $\mathrm{II}$ must be one of the following:
\begin{enumerate}
\item $D$ is a smooth elliptic curve.
\item $D$ is a rational curve with one node.
\item $D$ is a union of smooth rational curves with simple normal crossings (snc) whose dual graph is a ``circle'', called a ``circular boundary'' by Iitaka. That is, we have $D = D_1 + D_2 + ... + D_n$ such that
\begin{equation}\label{LOGK3RATE103}
\begin{aligned}
D_i(D-D_i) &= 2 \text{ for all } i\\
D_i D_j &= 0 \text{ for } i-j \not\equiv 0,\pm 1\ (\text{mod}\ n). 
\end{aligned}
\end{equation}
We call such $D$ a {\em circular boundary} of type $(\lambda_1, \lambda_2, ...,\lambda_n)$
if $D_i^2 = \lambda_i$.
\end{enumerate}

For a log surface $(X,D)$ with $X$ smooth and $D$ a nc divisor, a
{\em canonical blowup} $f: (\widehat{X}, \widehat{D}) \to (X, D)$ is the blowup of
$X$ at a singular point $p\in D_\text{sing}$ of $D$ with $\widehat{D} = f^{-1}(D)$
and a {\em canonical blowdown} $g: (X, D) \to (\overline{X}, \overline{D})$ is the contraction of a $(-1)$-curve contained in $D$ with $\overline{D} = g_* D$.

\subsection*{Acknowledgment}

The authors would like to thank David McKinnon for suggesting the arithmetic application of our result.

\section{Proof of Theorem \ref{LOGK3RATTHMLOGISOMK3}}\label{LOGK3RATSEC003}

The key construction here is a ``pivot operation'', which is also needed in the proof of our main theorem.

\begin{proof}[Proof Theorem \ref{LOGK3RATTHMLOGISOMK3}]
	We use the notation $\mu(G)$ to denote the number of irreducible components in a curve $G$. We will argue by induction on
	$\mu(D)$.
	
	If $\mu(D) = 1$, $D$ must be a smooth elliptic curve or a nodal rational curve and we have C0 or C1.
	
	Suppose that $D=D_1 + D_2 +...+D_n$ is a circular boundary of type
	$(\lambda_1, \lambda_2, ..., \lambda_n)$ with $D_i^2 = \lambda_i$.
	If $D$ contains a $(-1)$-curve $D_i$, we simply let $\pi: X\to \overline{X}$ be the contraction of $D_i$.
	Obviously, $\pi$ is a log isomorphism and we have reduced $\mu(D)$ by $1$.
	Suppose that $\lambda_i \ne -1$ for all $i$.
	
	Suppose that $\mu(D) = 2$. If $\lambda_1 = 0$ or $\lambda_2 = 0$,
	we have C3. Suppose that $\lambda_i \ne 0$. If $\lambda_1\le -2$ or $\lambda_2 \le -2$,
	we have C2. Otherwise, $\lambda_1, \lambda_2 > 0$ and we have C4.
	
	Suppose that $\mu(D) = n \ge 3$.
	If $\lambda_i\le -2$ for all but one $i$, we are done. Let us assume that there are at least two nonnegative $\lambda_i$'s.
	
	Suppose that one of $\lambda_i$'s vanishes, say $\lambda_1 = 0$.
	We have a log isomorphism $\pi: (X,D) \dashrightarrow
	(\overline{X}, \overline{D})$ composed of a blowup of
	$X$ at $D_1\cap D_2$ followed by a blowdown of the proper transform of $D_1$. On $\overline{X}$, we have
	$\overline{D}_n^2 = \lambda_n + 1$, $\overline{D}_1^2 = 0$
	and $\overline{D}_2^2 = \lambda_2 - 1$ if $n\ge 3$. We call such $\pi$ a {\em pivot} at $D_1$ (see Figure \ref{LOGK3RATFIGB2B3001}).
	When $n\ge 3$, 
	applying a sequence of pivot operations at $D_1$, we arrive at $(X,D)$ 
	with $D$ a circular boundary of type $(0, -1, \lambda_3, ...,\lambda_{n-1}, \lambda_n + \lambda_2 + 1)$; we then contract $D_2$, which will reduce $\mu(D)$ by $1$.
	
	\begin{figure}
		\begin{tikzpicture}[scale=1]
		\draw[thick] (-0.2, 0) -- (1.2, 0);
		\node[below] at (0.5,0) {{\tiny $\lambda_n$}};
		\node[below] at (0.5,-0.4) {{\small $(X,D)$}};
		\draw[thick] (0,-0.2) -- (0,1.2);
		\node[left] at (0,0.5) {{\tiny $0$}};
		\draw[thick] (-0.2,1) -- (1.2,1);
		\node[above] at (0.5,1) {{\tiny $\lambda_2$}};
		
		\draw[->,thick] (1.8,1.8) -- (1.2,1.2);
		
		\draw[thick] (1.8, 2) -- (3.2, 2);
		\node[below] at (2.5,2) {{\tiny $\lambda_n$}};
		\draw[thick] (2,1.8) -- (2,3.2);
		\node[left] at (2,2.5) {{\tiny $-1$}};
		\draw[thick] (2.4,3.8) -- (3.8,3.8);
		\node[above] at (3.2,3.8) {{\tiny $\lambda_2-1$}};
		\draw[thick] (1.88,2.84) -- (2.72,3.96);
		\node[left] at (2.3,3.5) {{\tiny $-1$}};
		
		\draw[->,thick] (4.2,1.8) -- (4.8,1.2);
		
		\draw[thick] (4.8, 0) -- (6.2, 0);
		\node[below] at (5.5,0) {{\tiny $\lambda_n+1$}};
		\node[below] at (5.5,-0.4) {{\small $(\overline{X},\overline{D})$}};
		\draw[thick] (4.88,-0.16) -- (5.72,0.96);
		\node[left] at (5.3,0.5) {{\tiny $0$}};
		\draw[thick] (5.4,0.8) -- (6.8,0.8);
		\node[above] at (6.2,0.8) {{\tiny $\lambda_2-1$}};
		
		\draw[->,dashed,thick] (2,0.5) -- (4,0.5);
		\end{tikzpicture}
		\caption{A pivot $\pi: (X, D) \dashrightarrow (\overline{X}, \overline{D})$ at $D_1$}
		\label{LOGK3RATFIGB2B3001}
	\end{figure}
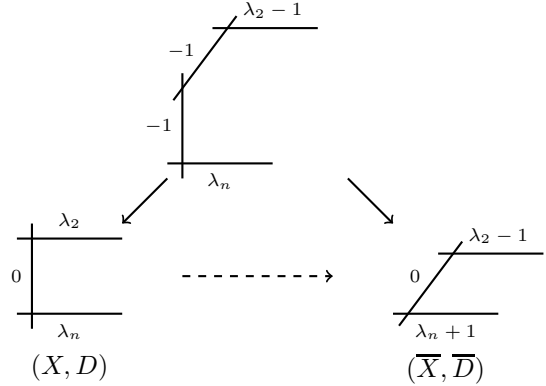
	
	Finally, we have the remaining case that $\lambda_i \ne 0, -1$ for all $i$, at least two of $\lambda_i$'s are positive and
	$\mu(D) \ge 3$. Suppose that $\lambda_i, \lambda_j > 0$ for some $i\ne j$. By Hodge index theorem and the fact that $D_i$ and $D_j$
	are linearly independent in $H^2(X,\BQ)$, we must have $D_i D_j = 2$
	and $n = 2$. Contradiction.
\end{proof}

\section{Irreducible boundary case}\label{LOGK3RATSEC002}

We are going to prove the following result in this section.

\begin{thm}\label{LOGK3RATTHM000}
For a genuine log K3 surface $(X, D)$ of type II where $D$ is either a smooth elliptic curve or a rational curve with one node,
there are countably many $\A^1$ curves in $X\backslash D$.
\end{thm}

Namely, we will prove that there are countably many $\A^1$ curves on a genuine log K3 of type C0 or C1.

Let us first revisit the following theorem of Geng Xu \cite{X}:

\begin{thm}[G. Xu]\label{LOGK3RATTHMGXU}
Given a smooth cubic curve $D$ in $\PP^2$, there are countably many rational curves in $\PP^2$ meeting $D$ set-theoretically at a unique point.
\end{thm}

\begin{proof}[Sketch of Xu's Proof]
It is easy to show that there are at most countably many rational curves meeting $D$ at a unique point. Roughly, if there is a complete one-parameter family of such rational curves, some fiber of the family must contain $D$, which is impossible.

Let $V_{A,g}$ be the Severi variety of integral curves in $|A|$ of genus $g$ and
$\overline{V}_{A,g}$ be its closure in $|A| = \PP H^0(A)$. It is well known that
\begin{equation}\label{LOGK3RATE001}
\dim V_{A,0} = A .D - 1 = a - 1.
\end{equation}
The key to produce infinitely many such rational curves is the following observation: For every ample divisor $A$ on $X$,
there exists a point $p\in D$ such that 
\begin{equation}\label{LOGK3RATE000}
ap = i_D^* A \text{ in } \Pic(D) \text{ and } mp \not\in i_D^* \Pic(X)
\text{ for all } 0 < m < a
\end{equation}
where $a = A.D$,
$i_D$ is the embedding $D\hookrightarrow X$ and 
$i_D^*:\Pic(X)\to \Pic(D)$ is the pullback between the Picard groups of $X$ and $D$.

Let $\Lambda\subset \PP H^0(A)$ be the subvariety consisting of $C\in |A|$ such that
$C$ meets $D$ at $p$ with multiplicity $a$. Then $\Lambda$ is a linear subspace of
$\PP H^0(A)$ of codimension $a-1$. So $\overline{V}_{A,0} \cap \Lambda \ne \emptyset$. Let $C\in \overline{V}_{A,0}\cap \Lambda$. Then every component of $C$ is rational and hence $D\not\subset C$. So $C$ meets $D$ properly at $p$ with multiplicity $a$. By our choice of $p$, $C$ must be integral.
\end{proof}

Note that the rational curves meeting $D$ set-theoretically at a single point are not necessarily $\A^1$ curves in $X\backslash D$: for $C\in \overline{V}_{A,0}\cap \Lambda$ in Xu's proof, there is no guarantee that $\nu^{-1}(p)$ consists of a single point on the normalization $C^\nu$ of $C$. Indeed, the computation of the corresponding Gromov-Witten invariants suggests that $\A^1$ curves form a proper subset of $\overline{V}_{A,0}\cap \Lambda$ \cite{T}.

So we need to adapt Xu's argument to $\A^1$ curves. There are two main ingredients of Xu's argument. One is \eqref{LOGK3RATE001}, which guarantees that there are ``sufficiently many'' rational curves on $X$. The other is \eqref{LOGK3RATE000}. His argument can be described by the phrase ``bend-and-not-break'': as he bends the rational curves
in $V_{A,0}$ to meet $D$ at $p$ with multiplicity $a$,
the condition \eqref{LOGK3RATE000} guarantees that the resulting curves do not break. Both \eqref{LOGK3RATE001} and \eqref{LOGK3RATE000} are also crucial to our argument. We have the following weak generalization of \eqref{LOGK3RATE000}.

\begin{lem}\label{LOGK3RATLEM000}
Let $D$ be a smooth elliptic curve or a nodal rational curve of arithmetic genus $p_a(D) = 1$ on a projective variety $X$ with the property that $i_D^* \Pic(X)$ is finitely generated over $\BZ$.
For every $A\in \Pic(X)$ with $a = AD\in \BZ^+$,
there exists a point
$p\in D_\text{sm}$ satisfying
\begin{equation}\label{LOGK3RATE013}
\begin{aligned}
ap &= i_D^* A \text{ in } \Pic(D) \text{ and}\\
mp &\not\in G = i_D^* \Pic(X)
\text{ for all } m \in \BZ^+ \text{ and } m < \sqrt{\frac{a}{|G_\text{tors}|}}
\end{aligned}
\end{equation}
where $G_\text{tors}$ is the torsion part of $G = i_D^* \Pic(X)$ and $D_\text{sm}$ is the smooth locus of $D$.
\end{lem}

\begin{proof}
Since torsions of all orders exist in $\Pic(D)$, we can find two points 
$p_1$ and $p_2$ on $D_\text{sm}$ such that $a p_1 = a p_2 = i_D^* A$ and $p_1-p_2$ is torsion
in $\Pic(D)$ of order exactly $a$. Suppose that \eqref{LOGK3RATE013} fails for both
$p_i$. Then there exist positive integers $k_1$ and $k_2$ such that
$k_1k_2 l < a$ and $k_i p_i\in G$ for $i=1,2$,
where $l = |G_\text{tors}|$. Then $k_1k_2(p_1 - p_2)\in G$ and $k_1k_2(p_1 - p_2)$
is torsion of order $\ge a/(k_1k_2) > l$. Contradiction.
\end{proof}

Now we are ready to prove Theorem \ref{LOGK3RATTHM000}.

\begin{proof}[Proof of Theorem \ref{LOGK3RATTHM000}]
It is well known that there are at most countably many $\A^1$ curves in $X\backslash D$ if $K_X + D$ is pseudo-effective (also see Lemma \ref{LOGK3RATLEMRIGIDITY} below).

Obviously, there exists a birational morphism
$g: X \to \overline{X}$ with $\overline{D} = g_* D$ such that
$\overline{X}$ is a minimal rational surface and $\overline{D}\in |-K_{\overline{X}}|$ is a smooth elliptic curve or a rational curve with one node. Indeed, $\overline{X}$ must be one of $\PP^2, \BF_0$ or $\BF_2$, where $\BF_\beta$ is the Hirzebruch surface $\PP(\CO_{\PP^1} \oplus \CO_{\PP^1}(\beta))$ over $\PP^1$. Let us replace $(X,D)$ by $(\overline{X}, \overline{D})$.

When $X\cong \PP^2$, we let $\pi: \widehat{X} \to X$ be the cyclic triple cover of $X$ ramified over $D$ and let
$\widehat{D} = \pi^{-1}(D)$. Clearly, if there are infinitely many $\A^1$ curves in $\widehat{X}\backslash \widehat{D}$, the same holds for $X\backslash D$.
Note that $\widehat{X}$ is a smooth cubic surface when $D$
is smooth and a cubic surface with an $A_2$ singularity when
$D$ is a nodal cubic. We replace $(X, D)$ by $(\widehat{X}, \widehat{D})$.

So in all these cases, we have a fiberation $f: X\to \PP^1$
whose general fibers are $\PP^1$. Let $C$ be a section and $F$ be a fiber of $f$. We choose $C$ such that $C\cap X_\text{sing} = \emptyset$.
Then every $\Gamma\in |C + mF|$ is supported on a union of smooth rational curves.

Note that the map
\begin{equation}\label{LOGK3RATE100}
\xymatrix{
	H^0(\CO_X(C+mF)) \ar@{->>}[r] & H^0(\CO_D(C + mF))
}
\end{equation}
is a surjection since $h^1(-D + C+mF) = h^1(-C-mF) = 0$.
Therefore, for every $p\in D_\text{sm}$ satisfying
$ap = C + mF$ in $\Pic(D)$, there exists $\Gamma\in |C+mF|$
such that $\Gamma$ meets $D$ at the unique point $p$.

By Lemma \ref{LOGK3RATLEM000}, there exists $p$ satisfying
\eqref{LOGK3RATE013}. So there exists a curve $\Gamma\in |C + mF|$
such that $\Gamma$ meets $D$ only at a point $p$ satisfying \eqref{LOGK3RATE013}. Let $\Gamma'$ be an irreducible component of $\Gamma$. Then $\Gamma'$ is a smooth rational curve meeting $D$ at the unique point $p$.
Note that $\Gamma'$ is Cartier since it is disjoint from
$X_\text{sing}$ when $X$ is a cubic surface with an $A_2$ singularity at the node of $D$. So
\begin{equation}\label{LOGK3RATE101}
\Gamma' D \ge \sqrt{\frac{a}{|G_\text{tors}|}}
\to\infty \text{ as } m\to\infty
\end{equation}
by \eqref{LOGK3RATE013}. Consequently, there are infinitely many $\A^1$ curves in $X\backslash D$.
\end{proof}

We are also ready to prove Theorem \ref{LOGK3RATTHMMCKINNON}. First, we need to justify the claim that every $\A^1$ curve in $X_\BC\backslash D_\BC$ is defined over $\overline{\BQ}$ for a log K3 surface $(X,D)$ over $\overline{\BQ}$.

\begin{lem}\label{LOGK3RATLEMRIGIDITY}
Let $D$ be an effective divisor of normal crossings on a smooth projective variety $X$. If $K_X + D$ is pseudo-effective,
there do not exist a quasi-projective variety $B$, a dominant morphism $f: Y = \PP^1\times B \to X$ and a section 
$\Gamma\subset Y$ of $Y/B$ such that
$f^{-1}(D) \subset \Gamma$. In addition, if $\dim X = 2$ and $(X, D)$ is defined over $\overline{\BQ}$, then every rational curve $C\subset X_\BC$ satisfying
$|\nu^{-1}(D)| \le 1$ is defined over $\overline{\BQ}$
with $\nu: C^\nu\to X$ the normalization of $C$.
\end{lem}

\begin{proof}
Suppose that such $f$ exists.
We may assume that $Y$ is smooth and $f$ is generically finite. 
Then $f^*(\Omega_X(\log D))\subset \Omega_Y(\log \Gamma)$.
It follows that $(K_Y + \Gamma) - f^*(K_X + D)$ is effective and hence $K_Y + \Gamma$ is pseudo-effective. But $(K_Y + \Gamma) . Y_b < 0$ for $b\in B$ general.
Contradiction.

Suppose that $\dim X = 2$, $(X,D)$ is defined over $\overline{\BQ}$
and $C\subset X_\BC$ is a rational curve satisfying $|\nu^{-1}(D)| \le 1$ and transcendental over $\overline{\BQ}$. Then by taking a spread, we can find a variety $B$ over $\overline{\BQ}$ of positive dimension and a non-trivial family $\C\subset X\times B$ of rational curves on $X$ such that $|\nu_b^{-1}(D)| \le 1$
for all $b\in B$, where $\nu: \widehat{\C}\to X\times B$ is the normalization of
$\C$. Contradiction.
\end{proof}

\begin{proof}[Proof of Theorem \ref{LOGK3RATTHMMCKINNON}]
In the proof of Theorem \ref{LOGK3RATTHM000}, we have actually found a sequence $\{\Gamma_n\}$ of rational curves on $X$ such that
each $\Gamma_n$ meets $D$ at a unique point $p_n\in D_\text{sm}$ with the properties
\begin{equation}\label{LOGK3RATE116}
\begin{aligned}
a_n p_n &\in G = i_D^* \Pic(X) \text{ for some } a_n\in \BZ^+\\
mp_n &\not\in G \text{ for all } 0 < m < a_n\\
\lim_{n\to\infty} a_n &= \infty.
\end{aligned}
\end{equation}
Let $M$ be the subgroup of $\Pic(D)$ generated by $p_n$. Then \eqref{LOGK3RATE116} implies that $M$ contains torsions of arbitrarily high orders.

Suppose that all $\Gamma_n$ are defined over a number field $k$.
WLOG, let us assume that $D$ is defined over $k$ as well.
Then $p_n$ are also defined over $k$. 
If $D$ is a smooth elliptic curve,
by Mordell-Weil (cf. \cite{S}), $M$ is finitely generated and cannot contain torsions of arbitrarily high orders.
If $D$ is a nodal rational curve, $M\subset k^*$ again cannot contain torsions of arbitrarily high orders in $(\BC^*)_\text{tors}$. Contradiction.
\end{proof}

\section{Proof of Theorem \ref{LOGK3RATTHM001}}\label{LOGK3RATSEC004}

\subsection{A necessary condition for the existence of infinitely many $\A^1$ curves}

\begin{lem}\label{LOGK3RATLEMB2}
Let $X$ be a smooth projective surface with $H^1(X) = 0$ and $D$ be a nc divisor on $X$. If there is an infinite sequence $\{C_m\subset X\}$ of integral curves of increasing degrees satisfying that
$|\nu_m^{-1}(D)| \le 1$ for the normalization $\nu_m: C_m^\nu \to X$
of $C_m$ and all $m$, then
\begin{equation}\label{LOGK3RATEB2}
\begin{aligned}
&\text{for every log isomorphism } f: \xymatrix{(X, D)\ar@{-->}[r]^-{\sim} & (\widehat{X}, \widehat{D})}
\text{ with}\\
&\quad \widehat{X} \text{ smooth and } \widehat{D} \text{ of nc},
\text{ either } \mu(\widehat{D}) = 1 \text{ or } 
\\
&\quad
\text{there exist a numerically effective (nef) and big divisor }
\widehat{L} \text{ on }\widehat{X}\\
&\quad
\text{and irreducible components }
\widehat{D}_1\ne \widehat{D}_2 \text{ of }
\widehat{D} \text{ such that}\\
&\quad\quad\widehat{D}_1\cap \widehat{D}_2\ne\emptyset
\text{ and }
\widehat{L}(\widehat{D} - \widehat{D}_1 - \widehat{D}_2) = 0.
\end{aligned}
\end{equation}
\end{lem}

\begin{proof}
Let $\widehat{C}_m$ be the proper transform of $C_m$ under $f$. Then $\widehat{C}_m$ meets $\widehat{D}$ at no more than one point.
Suppose that $\mu(\widehat{D}) > 1$. 
Since $\widehat{D}$ is a nc divisor, no three components of $\widehat{D}$ meet at one point.
Therefore, there exist components $\widehat{D}_1 \ne \widehat{D}_2$ 
of $\widehat{D}$ such that $\widehat{D}_1 \cap \widehat{D}_2 \ne \emptyset$ and $\widehat{C}_m\cap \widehat{D}_i = \emptyset$ for components $\widehat{D}_i\ne \widehat{D}_1, \widehat{D}_2$
and infinitely many $m$. Hence
\begin{equation}\label{LOGK3RATE104}
\widehat{C}_m (\widehat{D} - \widehat{D}_1 - \widehat{D}_2) = 0
\end{equation}
for infinitely many $m$.
We may simply assume that \eqref{LOGK3RATE104} holds for all $m$.

Next, we claim that there exists a nef and big divisor
$\widehat{L} = \sum a_i \widehat{C}_i$ for some $a_i\in \BZ$. This, combining with \eqref{LOGK3RATE104}, will imply
$\widehat{L} (\widehat{D} - \widehat{D}_1 - \widehat{D}_2) = 0$.

Obviously, $\widehat{C}_1, \widehat{C}_2, ..., \widehat{C}_m$ are linearly dependent
in $H^2(X,\BQ)$ as long as $m > h^2(X)$. So there exist integers $a_1, a_2, ..., a_m$, not all zero, such that
\begin{equation}\label{LOGK3RATE504}
a_1 \widehat{C}_1 + a_2\widehat{C}_2 + ... + a_m \widehat{C}_m = 0
\end{equation} 
in $H^2(X,\BQ)$. We write
\begin{equation}\label{LOGK3RATE505}
A = \sum_{a_i > 0} \widehat{C}_i \sim_\text{num} -\sum_{a_i < 0} a_i\widehat{C}_i = B.
\end{equation}
Since $\widehat{C}_i$ are effective, $A\ne 0$ and $B\ne 0$.

Clearly, $AC = BC \ge 0$ for all irreducible curves $C$ and thus $A$ and $B$ are nef. If $A^2 > 0$, $A$ is big and nef and we are done.
Suppose that $A^2 = 0$. If $A \widehat{C}_n > 0$ for some $n\in \BZ^+$, then $NA + \widehat{C}_n$ is big and nef for some $N>>1$. Suppose that $A\widehat{C}_n = 0$ for all $n$.

Since $H^1(X) = 0$, $A = B$ in $\Pic_\BQ(X)$. WLOG, suppose that
$A = B$ in $\Pic(X)$. Then $A$ and $B$ span a base-point-free pencil in $|A|$ which induces a map $f: X\to \PP^1$. Since $A\widehat{C}_n = 0$, each $\widehat{C}_n$ is containing in a fiber of $f$. This is impossible since $\deg \widehat{C}_n\to \infty$ as $n\to\infty$.
\end{proof}

Now we can prove the ``only if'' part of Theorem \ref{LOGK3RATTHM001}. That is, if there are infinitely many $\A^1$ curves in $X\backslash D$, $(X,D)$ cannot be of type C4.

\begin{proof}[Proof of ``only if'' part of Theorem \ref{LOGK3RATTHM001}]
If $(X,D)$ is C4, then $D = D_1 + D_2$ with $D_i^2 = \lambda_i > 0$.
Let $f: \widehat{X} \to X$ be the blowup of $X$ at the two intersections $D_1\cap D_2$; then $\widehat{D} = f^{-1}(D)$ has four components, each having self-intersection $\ge -1$, which obviously violates \eqref{LOGK3RATEB2} (see Figure \ref{LOGK3RATFIGB2B3000}).

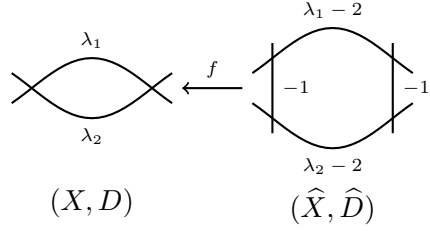
\begin{figure}
\begin{tikzpicture}[scale=0.8]
\draw[thick, domain=-30:210] plot (\x/90, {sin(\x)/2});
\draw[thick, domain=-30:210] plot (\x/90, {-sin(\x)/2});
\draw[thick, domain=-30:210] plot (\x/90 + 4, {1/2+sin(\x)/2});
\draw[thick, domain=-30:210] plot (\x/90 + 4, {-1/2 - sin(\x)/2});
\draw[thick] (4, 0.75) -- (4, -0.75);
\draw[thick] (6, 0.75) -- (6, -0.75);
\draw[->, thick] (3.5,0) -- (2.5, 0);
\node[below] at (1,-0.5) {{\tiny $\lambda_2$}};
\node[above] at (1,0.5) {{\tiny $\lambda_1$}};
\node[below] at (5,-1) {{\tiny $\lambda_2-2$}};
\node[above] at (5,1) {{\tiny $\lambda_1-2$}};
\node[above] at (3,0) {{\tiny $f$}};
\node[right] at (4,0) {{\tiny $-1$}};
\node[right] at (6,0) {{\tiny $-1$}};
\node[below] at (1,-1.5) {{$(X,D)$}};
\node[below] at (5,-1.5) {{$(\widehat{X},\widehat{D})$}};
\end{tikzpicture}
\caption{$D = D_1 + D_2$, $D_1^2 = \lambda_1 > 0$ and $D_2^2 = \lambda_2 > 0$}
\label{LOGK3RATFIGB2B3000}
\end{figure}
\end{proof}

\subsection{Infinitely many $\A^1$ curves on $(X,D)$ of type C3}

It remains to prove the ``if'' part of Theorem \ref{LOGK3RATTHM001}.
That is, there are infinitely many $\A^1$ curves on $(X,D)$ of type
C0-C3. We have prove the existence for C0 and C1 in Theorem \ref{LOGK3RATTHM000}. Type C3 is more or less trivial by the following lemma.

\begin{lem}\label{LOGK3RATLEM002}
Let $(X,D)$ be a genuine log K3 surface with circular boundary $D = D_1 + D_2$ satisfying that $D_2^2 = 0$.
Then there are infinitely many $\A^1$ curves in $X\backslash D$.
\end{lem}

\begin{proof}
Let $\pi: X\to \PP^1$ be the fiberation given by $|D_2|$. Obviously, $\pi$ factors through a ruled surface $\overline{X}$. Clearly, $\overline{X}$ is either $\BF_0$
or $\BF_1$. Let us replace $(X,D)$ by $(\overline{X}, \overline{D})$ with $\overline{D} = \pi_* D$.

So $X \cong \BF_\beta$, $D_1\in |2C + (\beta + 1) F|$ and $D_2\in |F|$,
where $\beta = 0$ or $1$, $F$ is a fiber of $\pi$ and $C$ is a section of $\pi$ with
$C^2 = -\beta$. Suppose that $D_1$ and $D_2$ meets at two points $p$ and $p'$.
For each $m\in \BN$, we have a smooth curve $C_m\in |C + (m+\beta) F|$ meeting $D_1$ at $p$ with multiplicity $2m + \beta + 1$. Clearly, $C_m\backslash \{p\}$ is an $\A^1$ curve in $X\backslash D$.
\end{proof}

\subsection{Infinitely many $\A^1$ curves on $(X,D)$ of type C2}

As a technical issue, we need $D_1^2 > 0$ later on. 
We can assume that by the following lemma.

\begin{lem}\label{LOGK3RATLEMD12}
Let $(X,D)$ be a genuine log K3 surface with circular boundary $D = D_1 + D_2 + ... + D_n$ satisfying $D_1^2 \ne 0,-1$ and $D_i^2 \le -2$
for $i\ne 1$. Then there exists a birational morphism
$f: X\to \overline{X}$ with $\overline{D} = f_* D$ such that
$(\overline{X}, \overline{D})$ is a genuine log K3 surface of either type C1 or type C2 with $\overline{D}_1^2 > 0$ and $\overline{D}_i^2 \le -2$ for $i\ne 1$.
\end{lem}

\begin{proof}
There is nothing to do if $D_1^2 > 0$. Let us assume that $D_i^2 \le -2$ for all $i$. We prove by induction on $\rank_\BZ \Pic(X)$.

Clearly, $X$ is not minimal. Let $f: X\to \overline{X}$ be the contraction of a $(-1)$-curve followed by a sequence of canonical blowdowns such that $\overline{D}$ has no component with self intersection $-1$. Clearly, all but one component of
$\overline{D}$ satisfy $\overline{D}_i^2 \le -2$. 
WLOG, suppose that $\overline{D}_i^2 \le -2$ for $i\ne 1$.
If $\mu(\overline{D}) = 1$ or $\overline{D}_i^2 > 0$, we are done.
If $\overline{D}_1^2 \le -2$, it follows from induction hypothesis since $f$ has reduced $\rank_\BZ \Pic(X)$ by $1$.

Suppose that $\overline{D}_1^2 = 0$. If $\mu(\overline{D}) \ge 3$, assuming $\overline{D}$ of type $(0, \lambda_2, ..., \lambda_m)$, then
a sequence of pivot operations at $\overline{D}_1$
\begin{equation}\label{LOGK3RATE506}
(0, \lambda_2, ..., \lambda_m)
\to (0, \lambda_2+1, ..., \lambda_m-1)
\to ... \to (0, -1, ..., \lambda_m + \lambda_2 + 1)
\end{equation}
followed by a sequence of canonical blowdowns will give us what we want.

Suppose that $\overline{D}_1^2 = 0$ and $\mu(\overline{D}) = 2$. If
$\overline{D}_1 E = 1$ for some $(-1)$-curve $E$, then blow down $E$ and we are done. Otherwise, $\overline{D}_1 E = 0$ for all $(-1)$-curves $E$ on $\overline{X}$. Then there exists a birational morphism $g: \overline{X} \to \widehat{X}$ with
$\widehat{D}_i = g_* \overline{D}_i$ such that $\widehat{D}_2^2 = -1$.
Blowing down $\widehat{D}_2$, we obtain a genuine log K3 of type C1.
\end{proof}

It remains to prove the following:

\begin{prop}\label{LOGK3RATPROPCB}
Let $(X, D)$ be a genuine log K3 surface with circular boundary $D = D_1 + D_2 + ... + D_n$ of type $(\lambda_1, \lambda_2, ..., \lambda_n)$. If $\lambda_1 > 0$ and
$\lambda_i \le -2$ for $i\ne 1$, then there are infinitely many $\A^1$ curves in
$X\backslash D$.
\end{prop}

We follow a similar line argument as Xu's proof of 
Theorem \ref{LOGK3RATTHMGXU}.
To start with, we need ``many'' rational curves in $X$ disjoint from $D_2 + D_3 + ... + D_n$, or equivalently, many rational curves in the smooth locus $\overline{X}_\text{sm}$ of $\overline{X}$, where $X\to \overline{X}$ is the contraction of $D_2 + D_3 + ... + D_n$.
This is where the theorem of Keel and McKernan comes in: the smooth locus of a log del Pezzo surface is rationally connected \cite[Corollary 1.6]{K-M}. 
We put their theorem in the following form:

\begin{thm}[Keel-McKernan]\label{LOGK3RATTHMLOGDELPEZZORATCON}
Let $X$ be a log del Pezzo surface. Then there exists an ample Cartier divisor $A$ on $X$ such that
\begin{itemize}
\item
$V_{A, 0}$ is nonempty of expected dimension $-K_X A - 1\ge 2$,
\item
a general member $C\in V_{A,0}$ lies inside $X_\text{sm}$,
\item
the normalization $\nu: C^\nu\to X$ of $C$ is an immersion,
\item
$\nu^* T_X$ is ample,
\item
and $C$ meets a fixed reduced curve $D$ transversely.
\end{itemize}
Furthermore, the same holds for $V_{mA,0}$ and all $m\in \BZ^+$
and a general member of $V_{mA,0}$ is nodal if $-mK_X A - 1\ge 4$.
\end{thm}

The key observation here is that once we can deform a rational curve away from $X_\text{sing}$,
the standard deformation theory of curves on smooth surfaces will take over. 
As long as $C\cap X_\text{sing} = \emptyset$ for a general member $C\in V_{A,0}$,
we can prove that $V_{A,0}$ has the expected dimension. In addition, as long as
$\dim V_{A,0} \ge 2$, a general member 
$C\in V_{A,0}$ behaves as expected (cf. \cite[Chapter 3, Section B]{H-M}), i.e.,
$\nu: C^\nu\to X$ is an immersion, $\nu^* T_X$ is ample
and $C$ is nodal if
$\dim V_{A,0}\ge 4$.
In our case, to deform a rational curve away from the only singularity of $X'$ or $\overline{X}$,
we actually only need a lemma in Keel-McKernan's paper \cite[Lemma 6.4]{K-M}.
Moreover, once we have $\dim V_{A,0} = -K_XA - 1$, we can produce more rational curves by taking two general members $C_1, C_2\in V_{A,0}$ and deforming the union $C_1\cup C_2$ to a rational curve in $V_{2A,0}$.
More generally, if $\dim V_{A_1,0} = -K_X A_1 - 1\ge 0$, $\dim V_{A_2,0} = -K_X A_2 -1\ge 0$ and two general members $C_1\in V_{A_1,0}$ and
$C_2\in V_{A_2,0}$ meet transversely, then $C_1\cup C_2$ can be deformed to a rational curve in $V_{A_1+A_2,0}$. So we can prove in this way that the theorem holds for
all $V_{mA,0}$. 

Basically, we want to impose tangency conditions on $C\in V_{A,0}$. 
Let us first define the subvarieties of Severi varieties of curves on $X$ tangent to a fixed curve $D$ as follows.

\begin{defn}\label{LOGK3RATDEFSEVERITANGENT}
For a curve $D$ on a projective surface $X$ and a zero cycle
$\alpha = m_1 p_1 + m_2 p_2 + ... + m_kp_k\in Z_0(D)$, we use the notation
$V_{A, g, D, \alpha}$ to denote the subvariety of $V_{A,g}$ consisting of
integral curves $C\in |A|$ of genus $g$ satisfying that
\begin{itemize}
\item
$C$ meets $D$ properly and
\item
there exists
$q_i\in \nu^{-1}(p_i)$ and $n_i\ge m_i$ such that
$q_1, q_2,...,q_k$ are distinct and
$\nu^* D = n_i q_i$ when $\nu$ is restricted to the open neighborhoods
of $p_i$ and $q_i$ for $i=1,2,...,k$,
\end{itemize}
where $\nu: \widehat{C}\to X$ is the normalization of $C$,
$m_1,m_2,...,m_k\in \BQ^+$ and $p_1,p_2,...,p_k$ are points on $D$ such that
$D$ is locally $\BQ$-Cartier at each $p_i$.
\end{defn}

We are going to prove Proposition \ref{LOGK3RATPROPCB} 
by showing that there are infinitely many rational curves $C\subset X$ meeting $D$ only at $p \in D_1\cap D_2$; more precisely,
we are going to show
\begin{equation}\label{LOGK3RATE138}
V_{A_m,0,D,a_m p} \ne \emptyset
\end{equation}
for a sequence of divisors $A_m$ satisfying $a_m = A_m D\to\infty$ as $m\to \infty$. For starters, we prove the following:

\begin{prop}\label{LOGK3RATPROPRATLOGDELPEZ}
Let $(X, D)$ be a genuine log K3 surface with circular boundary
$D = D_1 + D_2 + ... + D_n$ of type $(\lambda_1, \lambda_2, ..., \lambda_n)$. If
$\lambda_1 > 0$ and $\lambda_i \le -2$ for $i\ne 1$, then
$(X,D)$ can replaced by a log isomorphic model such that
$D_1^2 > 0$,
the intersection matrix of $D-D_1$ is negative definite and there exist a sequence of divisors $A_m$ on $X$ satisfying that
\begin{equation}\label{LOGK3RATE132}
\begin{aligned}
A_m(D - D_1) &= A_m D_2 = 1,\\
\lim_{m\to\infty} A_m D &= \infty,\\
\dim V_{A_m, 0, D, 2p} &= A_m D - 2 \text{ for } p \in D_1\cap  D_2,\\
&\hspace{-96pt} \text{and a general member } C_m\in V_{A_m, 0, D, 2p}
\text{ meets } D\\
&\text{transversely at } A_m D - 2 \text{ points outside of } p.
\end{aligned}
\end{equation}
\end{prop}

\begin{proof}
Let us first prove that $(X,D)$ can replaced by a log isomorphic model such that
$D_1^2 > 0$,
the intersection matrix of $D-D_1$ is negative definite and there is an effective divisor $F$ on $X$ such that $F^2 = 0$ and $FD_1 = FD_2 = 1$.

If $n=2$, then there exists a fiberation $\pi: X\to \PP^1$ whose general fibers are $\PP^1$ and a fiber $F$ of $\pi$ has the required property. Suppose that $n\ge 3$. Then there exists a $(-1)$-curve $E$ such that $D_k E = 1$ for some $k\ne 1$. We prove that there exists a log isomorphism $f: (X,D) \dashrightarrow (\widehat{X}, \widehat{D})$ such that
the proper transform of $E$ is the divisor $F$ we want. This $f$ is given by a sequence of canonical blowups and blowdowns and pivot operations. First, we can replace $D_1$ by a chain of curves of self intersections $(-2,-2,...,-2,-1,0)$ by a sequence of canonical blowups over $D_1\cap D_n$:
\begin{equation}\label{LOGK3RATE110}
\begin{aligned}
&\quad (\lambda_1,\lambda_2, ..., \lambda_n)\\
&\to (\lambda_1-1,\lambda_2, ..., \lambda_n-1,-1)\\
&\to (\lambda_1-2,\lambda_2, ..., \lambda_n-1,-2,-1) \to ...\\
&\to (0,\lambda_2, \lambda_3, ..., \lambda_n-1,-2,-2, ..., -2, -1).
\end{aligned}
\end{equation}
Then a sequence of pivots at $D_1$ render $D_2^2=-1$, $D_2$ can then be contracted and $D_1^2$ is restored to $0$ by a canonical blowup:
\begin{equation}\label{LOGK3RATE502}
\begin{aligned}
&\to (0, -1, \lambda_3, ..., \lambda_n-1,-2,-2, ..., -2,\lambda_2)\\
&\to (1, \lambda_3+1, \lambda_4, ..., \lambda_n-1,-2,-2, ..., -2,\lambda_2)\\
&\to (0, \lambda_3+1, \lambda_4, ..., \lambda_n-1,-2,-2, ..., -2,\lambda_2-1,-1).
\end{aligned}
\end{equation}
We continue this process until $D_k$ is contracted:
\begin{equation}\label{LOGK3RATE503}
\begin{aligned}
&\to (0, -1, \lambda_4, ..., \lambda_n-1,-2,-2, ..., -2,\lambda_2-1,\lambda_3+1)\\
&\to (1, \lambda_4+1, ..., \lambda_n-1,-2,-2, ..., -2,\lambda_2-1,\lambda_3+1)\\
&\to (0, \lambda_4+1, ..., \lambda_n-1,-2,-2, ..., -2,\lambda_2-1,\lambda_3, -1)\\
&\to (0, -1, \lambda_5, ..., \lambda_n-1,-2, ..., -2,\lambda_2-1,\lambda_3, \lambda_4+1)\\
&\to (1, \lambda_5+1, ..., \lambda_n-1,-2, ..., -2,\lambda_2-1,\lambda_3, \lambda_4+1)\to ...\\
&\to (1, \lambda_k+1, ..., \lambda_n-1,-2, ..., -2,\lambda_2-1,\lambda_3, ..., \lambda_{k-2},\lambda_{k-1}+1)\\
&\to (0, \lambda_k+1, ..., \lambda_n-1,-2, ..., -2,\lambda_2-1,\lambda_3, ..., \lambda_{k-2},\lambda_{k-1},-1)\\
&\to (0, -1, \lambda_{k+1}, ..., \lambda_n-1,-2, ..., -2,\lambda_2-1,\lambda_3, ...,\lambda_{k-1},\lambda_k + 1)\\
&\to (1, \lambda_{k+1}+1, ..., \lambda_n-1,-2, ..., -2,\lambda_2-1,\lambda_3, ...,\lambda_{k-1},\lambda_k + 1)
\end{aligned}
\end{equation}
where \eqref{LOGK3RATE110}-\eqref{LOGK3RATE503} illustrate how the type of circular boundary $D$ changes in the process. At the last step, when we contract the proper transform $\widehat{D}_k$ of $D_k$, the self-intersection $E^2$ of $E$ increases by $1$ and its proper transform $F$ is what we are after.

Applying Theorem \ref{LOGK3RATTHMLOGDELPEZZORATCON} to the log del Pezzo surfaces $\overline{X}$ obtained from
$X$ by contracting $D - D_1$, we obtain base-point-free (bpf) and
big divisors $A$ on $X$ such that
\begin{equation}\label{LOGK3RATE133}
\begin{aligned}
A (D - D_1) &= 0,\\
\dim V_{A,0} &= A D - 1 \ge 2.
\end{aligned}
\end{equation}
Let $A_m = mA + F$. Clearly, $A_m D_2 = 1$, $A_m D_i = 0$ for $i\ne 1,2$ and $A_m D\to\infty$ as $m\to\infty$.

Let $F_p$ be the member of the pencil $|F|$ passing through $p$. Then the union
$\Gamma_1\cup \Gamma_2\cup ... \cup \Gamma_m\cup F_p$ can be deformed to a curve $C_m\in V_{A_m,0,D,2p}$ for $m$ general members $\Gamma_i\in V_{A,0}$. So
\begin{equation}\label{LOGK3RATE140}
\dim V_{A_m,0,D,2p} = A_m D - 2.
\end{equation}
Using the standard deformation theory and the rigidity lemma \ref{LOGK3RATLEMRIGIDITY}, it is easy to see that
a general member $C_m\in V_{A_m,0,D,2p}$ meets $D$ transversely
at $A_m D - 2$ points outside of $p$.
\end{proof}

Starting with \eqref{LOGK3RATE140},
naturally, we try to prove \eqref{LOGK3RATE138} by imposing more tangency conditions on $C_m\in V_{A_m,0,D,2p}$ at $p$.
We are going to do this inductively by increasing the multiplicity at $p$ one at a time. That is, we will roughly show that
\begin{equation}\label{LOGK3RATE141}
\dim V_{A_m,0,D,kp} = A_m D - k
\end{equation}
for all $k$.
When we deform/degenerate a family of rational curves on $X$ for this purpose, one difficulty arises: its flat limit might contain some components of $D$. To deal with this situation, we need the following key lemma. 

\begin{lem}\label{LEMK3RATKEYLEMMA}
Let $X$ be a smooth projective surface,
$D = D_1 + D_2 + ... + D_n$ be a circular boundary on $X$ and
$f: Y/\Delta\to X$ be a family of stable rational maps
over the unit disk $\Delta = \{|t| < 1\}$ satisfying that $Y_t\cong \PP^1$ and
$f(Y_t)$ meets $D$ properly for $t\ne 0$.
Suppose that
\begin{equation}\label{LOGK3RATE117}
\begin{aligned}
\dim (D \cap f(Y_0)) &> 0,\\
f^* D &= \sum_{i=1}^c m_i \Gamma_i + V \text{ and}\\
D_\text{sing} \cap f(\Gamma_i) &= \emptyset\text{ for } i\ne 1,2,
\end{aligned}
\end{equation}
where $V\subset Y_0$ and $\Gamma_i$ are distinct sections of
$Y/\Delta$ of multiplicities $m_i > 0$ in $f^*D$.

Let $G = \Gamma_1\cup \Gamma_2\cup \supp(V)\subset Y$ and
let $G^\circ$ be the curve obtained from $G$ by contracting all contractible components under the map $f: G\to D$. That is, $G \to G^\circ \to D$ is the Stein factorization of
$f: G\to D$. Then $G^\circ$ is a chain of curves given by (see Figure \ref{LOGK3RATFIGKEYLEMMA} for a configuration of $G^\circ$ when $n=2$)
\begin{equation}\label{LOGK3RATE106}
G^\circ = \Gamma_1\cup C_1\cup C_2 \cup \ldots \cup C_a \cup \Gamma_2
\end{equation}
where $\Gamma_1 \cap C_1 = q_0$, $C_i\cap C_{i+1} = q_{i}$,
$C_a\cap \Gamma_2 = q_a$, 
\begin{itemize}
\item $f(q_i) \in D_\text{sing}$ for $i=0,1,...,a$,
\item $f$ sends each $C_i$ onto one of $D_j$ with a map totally ramified over 
the two intersections $D_j\cap (D-D_j)$ for $i=1,2,...,a$,
\item $f$ maps $G^\circ$ locally at $q_i$ surjectively onto $D$ at $f(q_i)$
for $0< i < a$,
\item $f$ maps $G^\circ$ locally at $q_0$ surjectively onto $D$ at $f(q_0)$
if $f_*\Gamma_1 \ne 0$,
\item and $f$ maps $G^\circ$ locally at $q_a$ surjectively onto $D$ at $f(q_a)$
if $f_*\Gamma_2 \ne 0$.
\end{itemize}
In particular, if $f_*\Gamma_1 \ne 0$ and $f(\Gamma_1)\subset D_1$, $\Gamma_1$ lies on the connected component $M$ of $f^{-1}(D_1)$ such that $\Gamma_i\not\subset M$
for all $i\ne 1$ and
$f_* E = 0$ for all irreducible components $E\ne \Gamma_1\subset M$.
\end{lem}

We call a curve $F = F_1\cup F_2\cup ...\cup F_n$ a {\em chain} of curves if the dual graph of $F$ is a chain, i.e., a tree with at most two vertices of degree $\le 1$.

\begin{figure}
\begin{tikzpicture}[scale=1]
\draw[thick, domain=-0.2:1.2] plot ({1*\x - \x*\x},\x);
\draw[thick, domain=-0.2:1.2] plot ({1*\x - \x*\x},\x+1);
\draw[thick, domain=-0.2:1.2] plot ({1*\x - \x*\x},\x-1);
\draw[thick, domain=-0.2:1.2] plot ({1*\x - \x*\x},\x-2);
\draw[thick, domain=-120:120] plot ({cos(\x) / 4 + 2} , \x/135 );
\draw[thick, domain=-120:120] plot ({-cos(\x) / 4 + 2} , \x/135 );
\draw[thick] (-1, -2) -- (0.3,-2);
\draw[thick] (-1, 2) -- (0.3,2);
\node[above] at (-1.1,2) {{$\Gamma_1$}};
\node[below] at (-1.1,-2) {{$\Gamma_2$}};
\node[below] at (2,-1.3) {{$D$}};
\node[right] at (0.1,-2.6) {{\tiny $\nu: \widehat{D}\to D$ the normalization}};
\node[left] at (0.3, 1.5) {{$C_1$}};
\node[left] at (0.3, 0.5) {{$C_2$}};
\node[left] at (0.3, -0.5) {{$C_3$}};
\node[left] at (0.3, -1.5) {{$C_4$}};
\node[above] at (2,0.75) {{\tiny $p$}};
\node[below] at (2,-0.75) {{\tiny $p'$}};
\node[above] at (0.1,2) {{\tiny $q_0$}};
\node[below] at (0.1,-2) {{\tiny $q_4$}};
\node[right] at (0,1) {{\tiny $q_1$}};
\node[right] at (0,0) {{\tiny $q_2$}};
\node[right] at (0,-1) {{\tiny $q_3$}};
\draw[->, thick] (0.5,0) -- (1.5,0);
\node[above] at (1,0) {{$f$}};

\draw[thick] (4, 2) -- (5.3,2);
\draw[fill] (5,2) circle [radius=0.05];
\node[above] at (5,2.1) {{\tiny $q_0$}};
\draw[thick, domain=-0.2:1.2] plot ({1*\x - \x*\x+5},\x+0.5);
\node[left] at (5.3, 1) {{$C_1$}};
\draw[fill] (5,1.5) circle [radius=0.05];
\draw[fill] (5,0.5) circle [radius=0.05];
\node[left] at (5,1.5) {{\tiny $q_0$}};
\node[left] at (5,0.5) {{\tiny $q_1$}};
\draw[->,thick] (5.1,1.5) -- (7.4,1.6667);
\draw[->,thick] (5.1,1.9) -- (7.35,-0.2533);
\draw[->,thick] (5.1,0.5) -- (7.4,0.3333);

\draw[thick, domain=-120:120] plot ({cos(\x) / 4 + 7.5} , \x/135 -1);
\node[right] at (7.75,-1) {{$D_1$}};
\draw[thick, domain=-120:120] plot ({-cos(\x) / 4 + 7.5} , \x/135 + 1);
\node[right] at (7.75,1) {{$D_2$}};

\node[above] at (3.9,2) {{$\Gamma_1$}};
\node[below] at (3.9,-2) {{$\Gamma_2$}};

\node[below] at (7.5,-2) {{$\widehat{D}$}};
\draw[fill] (7.5,-0.3333) circle [radius=0.05];
\draw[fill] (7.5,-1.6667) circle [radius=0.05];
\node[right] at (7.5,-0.3333) {{\tiny $p$}};
\node[right] at (7.5,-1.6667) {{\tiny $p'$}};
\draw[fill] (7.5,1.6667) circle [radius=0.05];
\draw[fill] (7.5,0.3333) circle [radius=0.05];
\node[right] at (7.5,1.6667) {{\tiny $p$}};
\node[right] at (7.5,0.3333) {{\tiny $p'$}};

\draw[thick] (4, -2) -- (5.3,-2);
\draw[fill] (5,-2) circle [radius=0.05];
\node[below] at (5,-2.1) {{\tiny $q_4$}};
\draw[->, thick] (5,-1.9) -- (7.3,1.5567);

\draw[thick, domain=-0.2:1.2] plot ({1*\x - \x*\x+5},\x-1.5);
\node[left] at (5.3, -1) {{$C_4$}};
\draw[fill] (5,-0.5) circle [radius=0.05];
\draw[fill] (5,-1.5) circle [radius=0.05];
\node[left] at (5,-0.5) {{\tiny $q_3$}};
\node[left] at (5,-1.5) {{\tiny $q_4$}};
\draw[->, thick] (5.1,-0.5) -- (7.4,-1.6667);
\draw[->, thick] (5.1,-1.5) -- (7.4,-0.3333);
\end{tikzpicture}
\caption{A configuration of $G^\circ$ for $f: G^\circ \to D = D_1 + D_2$}
\label{LOGK3RATFIGKEYLEMMA}
\end{figure}
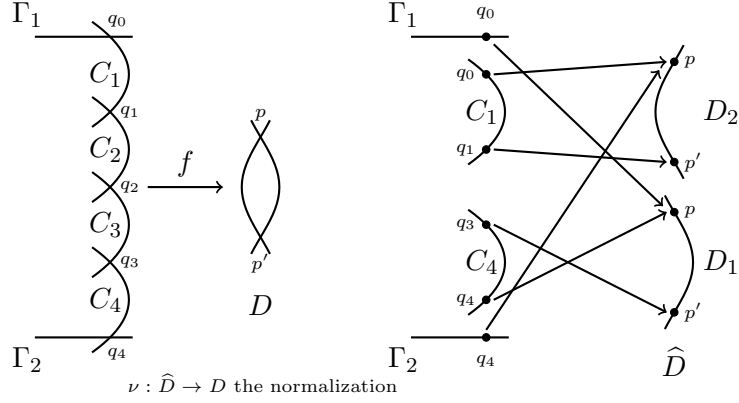

\begin{proof}[Proof of Lemma \ref{LEMK3RATKEYLEMMA}]
We first prove the following statement:

\begin{claim}\label{LOGK3RATCLAIMKEYLEMMA}
For every component $C\subset Y_0$ and a point $q\in C$ satisfying that $f_*C \ne 0$,
$f(C) \subset D$ and $f(q) \in D_\text{sing}$,
there exists either a chain $C\cup E_1\cup E_2\cup ...
\cup E_a\cup \Gamma_i$ of curves satisfying
\begin{equation}\label{LOGK3RATE118}
\begin{aligned}
& C\cap E_1 = q, E_k\cap E_{k+1} \ne \emptyset,
f_* E_k = 0, E_a\cap \Gamma_i = q' \text{ for some }
\Gamma_i\\
&\hspace{24pt} \text{and } f \text{ maps }
C \text{ and } \Gamma_i \text{ locally at }
q \text{ and } q'\\
&\hspace{36pt}\text{to the two branches of }
D \text{ at } f(q), \text{ respectively, if } f_* \Gamma_i \ne 0 
\end{aligned}
\end{equation}
or a chain $C\cup E_1\cup E_2\cup ... \cup E_a\cup C'$ of curves satisfying
\begin{equation}\label{LOGK3RATE119}
\begin{aligned}
& C\cap E_1 = q, E_k\cap E_{k+1} \ne \emptyset,
f_* E_k = 0, E_a\cap C' = q'\\
& \text{ for some component }
C'\subset Y_0 \text{ with } f_*(C') \ne 0 \text{ and } f(C') \subset D\\
&\hspace{48pt}\text{and } f \text{ maps }
C \text{ and } C' \text{ locally at }
q \text{ and } q'\\
&\hspace{72pt}\text{to the two branches of }
D \text{ at } f(q), \text{ respectively}.
\end{aligned}
\end{equation}
\end{claim}

WLOG, we assume that $f(C) = D_1$ and $f(q) = p = D_1\cap D_2$.
The statement is local. So we choose an analytic open neighborhood $U$ of $p\in X$ and
let $M$ be the connected component of $f^{-1}(U)$ that contains
$q$. Since $q\in f^{-1}(D_2)$, we have either
$\Gamma_i \cap M \ne \emptyset$ for some $\Gamma_i$
with $f(\Gamma_i\cap M) \subset D_2$
or
$C'\cap M \ne \emptyset$ for some component $C'\subset Y_0$
with $f(C'\cap M) = D_2\cap U$ such that $\Gamma_i$ or $C'$ is joined to $C$ by a chain of contractible components. In addition, if it is the former and $f_*\Gamma_i \ne 0$,
we necessarily have $f(\Gamma_i\cap M) = D_2\cap U$.
This proves Claim \ref{LOGK3RATCLAIMKEYLEMMA}.

We let $\Sigma$ be the subgraph of the dual graph of $Y_0$ that contains all components $C\subset Y_0$ satisfying $f_* C \ne 0$ and
$f(C) \subset D$ and all chains $C\cup E_1\cup E_2\cup ... \cup E_a\cup C'$ with the property \eqref{LOGK3RATE119}. Note that every contractible component in $\Sigma$ has degree $\ge 2$.

If $D_\text{sing}\cap f(\Gamma_i) = \emptyset$ for all but one $\Gamma_i$, then
by Claim \ref{LOGK3RATCLAIMKEYLEMMA},
all but one vertices in $\Sigma$ have degree $\ge 2$
with the remaining vertex of degree $\ge 1$,
which contradicts the fact that $\Sigma$ is a disjoint union of trees.

Therefore, $D_\text{sing}\cap f(\Gamma_i)\ne\emptyset$ for $i = 1,2$ and
all but two vertices in $\Sigma$ have degree $\ge 2$
and the remaining two vertices have total degree $\ge 2$.
So $\Sigma$ has to be a chain. Then it is easy to see that
$G^\circ$ is a chain of curves with the properties described by the lemma.
\end{proof}

To finish the proof of Proposition \ref{LOGK3RATPROPCB} and thus settle the last case of Theorem \ref{LOGK3RATTHM001}, it remains to prove the following:

\begin{prop}\label{LOGK3RATPROPTANGENTD1D2}
Let $X$ be a smooth projective surface, $D = \sum_{i=1}^n D_i$ be a circular boundary on $X$, $p\in D_1\cap D_2$ and $A$ be a divisor on $X$ satisfying that 
$a = AD = A(D_1+D_2) > AD_2 = 1$.
If $K_X + D$ is pseudo-effective, $D_1^2 > 0$, $\dim V_{A, 0, D, 2p} = a - 2$ and a general member of $V_{A, 0, D, 2p}$ meets $D$ transversely at $a-2$ points outside of $p$, then
for each $0\le l \le a-2$,
there exist $m_0, m_1, ..., m_{a-l-2}\in \BZ^+$ and $A_l \le A$ such that
\begin{equation}\label{LOGK3RATE112}
\begin{aligned}
& a = A_l D = 1 + \sum_{i=0}^{a-l-2} m_i \text{ and}
\\
& V_{A_l, 0, D, \alpha} \ne \emptyset \text{ for }
\alpha = (m_0+1) p + \sum_{i=1}^{a-l-2} m_i p_i
\end{aligned}
\end{equation}
where $p_1, p_2, ..., p_{a-l-2}$ are $a-l-2$ general points on $D_1$ and
we write $A \ge B$ for two divisors $A$ and $B$ if $A-B$ is effective.
\end{prop}

\begin{proof}
Since a general member of $V_{A, 0, D, 2p}$ meets $D$ transversely at $a-2$ points outside of $p$, we have
\begin{equation}\label{LOGK3RATE114}
V_{A,0,D,\alpha} \ne\emptyset \text{ for } \alpha = 2p + p_1 + p_2 + ... + p_{a-2}
\end{equation}
where $p_1, p_2, ..., p_{a-2}$ are $a-2$ general points on $D_1$. So the proposition holds for $l=0$. We argue by induction on $l$.

Suppose that there exist $m_0, m_1, ..., m_{a-l-2}\in \BZ^+$ such that
\begin{equation}\label{LOGK3RATE120}
\dim V_{A, 0, D, \alpha} = 0 \text{ for }
\alpha = (m_0+1) p + \sum_{i=1}^{a-l-2} m_i p_i.
\end{equation}
Among $p_i$, we fix $a-l-3$ general points $p_2, p_3, ..., p_{a-l-2}$ on $D_1$, let
\begin{equation}\label{LOGK3RATE115}
\lambda = \sum_{i=2}^{a-l-2} m_i p_i
\end{equation}
and let $q = p_{1}$ vary. More precisely, we consider the closure $W$ of
\begin{equation}\label{LOGK3RATE107}
\begin{aligned}
W^\circ &= \{ (C, q): C\in V_{A,0,D,(m_0+1) p+ m_1 q + \lambda} \text{ and } q\in D_1 \text{ general} \}\\
&\subset |A| \times D_1.
\end{aligned}
\end{equation}
By induction hypothesis \eqref{LOGK3RATE120}, $W$ is finite over $D_1$.

Let us consider the stable maps associated to the family of curves $C_q\in W_q$ over $q\in D_1$.
That is, there exists a finite morphism $\phi: B\to W\to D_1$ and
a family $f: Y/B\to X$ of stable rational maps
satisfying that $f_* Y_b \in W_{\phi(b)}$ for all $b\in B$.

Obviously,
\begin{equation}\label{LOGK3RATE111}
f^*D_1 = m_0 P + m_1 Q + \sum_{i=2}^{a-l-2} m_i P_i + V = m_0 P + m_1 Q + \Lambda + V
\end{equation}
where $\pi_* V = 0$ for $\pi:Y\to B$ and $P, P_i$ and $Q$ are the sections of $Y/B$ satisfying that
$f(P) = p$, $f(P_i) = p_i$ and $f(Q\cap Y_b) = \phi(b)$ for all $b\in B$.

In other words, $Q$ is the moving intersections
between $f_* Y_b$ and $D_1$, while $P$ and $P_i$ are the fixed intersections. We want to show that $Q$ ``collides'' with one of $P$ and $P_i$, which will reduce the number of points in $\{p_1,p_2,...,p_{a-l-2}\}$ and thus increase $l$ by one.

One of the key hypotheses is $D_1^2 > 0$. So $D_1$ is nef and big. Consequently,
$f^{-1}(D_1)$ is connected. So $Q$ and one of $P$ and $P_i$ are joined by a chain of curves in $V$. More precisely, either $P + V_0 + Q$ or $P_i + V_0 + Q$ is connected for some $i$ and a connected component $V_0$ of $V$ contained in a fiber $Y_b$. We will be almost done if
$f(Y_b)$ meets $D$ properly, i.e.,
\begin{equation}\label{LOGK3RATE105}
\dim (f(Y_b)\cap D) = 0,
\end{equation}
which implies that $f_* V_0 = 0$.
This is guaranteed by our key lemma: If $f(Y_b)$ fails to meet $D$ properly, then since $f(Q) = D_1$, by Lemma \ref{LEMK3RATKEYLEMMA},
$Q$ lies on a connected component $M$ of $f^{-1}(D_1)$ such that
$P, P_i\not\subset M$ in an analytic open neighborhood of $Y_b$. Contradiction.
So we necessarily have \eqref{LOGK3RATE105}.

Therefore, $C = f_* Y_b$ meets $D$ properly and $f_* V_0 = 0$. Hence
\begin{equation}\label{LOGK3RATE121}
C . D = (m_0 + m_1 + 1) p + \sum_{j=2}^{a-l-2} m_j p_j
\end{equation}
if $P + V_0 + Q$ is connected and
\begin{equation}\label{LOGK3RATE123}
C . D = (m_0 + 1) p + m_1 p_i + \sum_{j=2}^{a-l-2} m_j p_j
\end{equation}
if $P_i + V_0 + Q$ is connected for some $2\le i\le a-l-2$.

We claim that we cannot write $C = C_1 + C_2$ with $C_k\ge 0$ and $C_k D > 0$.
Otherwise, since $C D_2 = 1$, one of $C_1$ and $C_2$ does not pass through $p$. Thus,
\begin{equation}\label{LOGK3RATE124}
\begin{aligned}
C_1 . D &= a_{11} p + \sum_{j=2}^{a-l-2} a_{1j} p_j\\
C_2 . D &= \sum_{j=2}^{a-l-2} a_{2j} p_j
\end{aligned}
\end{equation}
for some $a_{1j}, a_{2j}\in \BN$. So $C_2$ meets $D$ at the smooth points on $D_1$ and hence
\begin{equation}\label{LOGK3RATE125}
\sum_{j=2}^{a-l-2} a_{2j} p_j = i_D^* C_2 \in i_D^*\Pic(X)
\end{equation}
in $\Pic(D)$. Note that $i_D^*\Pic(X)$ is a finitely generated subgroup of $\Pic(D)$. On the other hand, $p_2, p_3, ..., p_{a-l-2}$ are general points on $D_1$. So \eqref{LOGK3RATE125} cannot hold. Therefore, we necessarily have
\begin{equation}\label{LOGK3RATE126}
C = \Gamma + E
\end{equation}
with $\Gamma$ integral, $E$ effective and $\Gamma D = A D = a$.
Clearly, $\Gamma\in V_{A_{l+1}, 0, D, \alpha}$ for $A_{l+1} = \Gamma$ and
\begin{equation}\label{LOGK3RATE127}
\begin{aligned}
\alpha &= (m_0 + m_1 + 1) p + \sum_{j=2}^{a-l-2} m_j p_j \text{ or}\\
\alpha &= (m_0 + 1) p + m_1 p_i + \sum_{j=2}^{a-l-2} m_j p_j
\end{aligned}
\end{equation}
for some $2\le i \le a-l-2$.
\end{proof}

\section{Iitaka Models}\label{LOGK3RATSECIITAKA}

\subsection{Iitaka models}

Iitaka had a complete classification of log K3's. More generally, Iitaka and Zhang also classified Iitaka surfaces, which are log K3's with the condition $h^0(\Omega_X(\log D)) = 0$ removed. For our purpose, we just need the following \cite[Theorem 3 \& Theorem $\text{II}_\text{a}$ \& Table $\text{II}_\text{a}$ \& Proposition 16
\& Table $\text{II}_\text{b}$]{I}

\begin{thm}[Iitaka's Classifications of Type II log K3]\label{LOGK3RATTHMIITAKA}
For every Type $\mathrm{II}$ log K3 $(X, D)$, there exists a birational morphism $f: X\to \overline{X}$, where $\overline{X}$ is a minimal rational surface,
$\overline{D} = f_* D$ is a nc divisor and
$(\overline{X}, \overline{D})$ is one of the following:
\begin{enumerate}
\item[(a-i)]
$(\PP^2, E)$ where $E$ is a smooth elliptic curve;
\item[(a-ii)]
$(\BF_0, E)$ where $E$ is a smooth elliptic curve;
\item[(a-iii)]
$(\BF_2, E)$ or $(\BF_2, E + \Delta_\infty)$, where $E$ is a smooth elliptic curve and $\Delta_\infty$ is the section of $\BF_2/\PP^1$ with $\Delta_\infty^2 = -2$;
\item[(b-i)]
$(\PP^2, H_1 + H_2 + H_3)$ where each $H_i$ is a line on $\PP^2$;
\item[(b-ii)]
$(\BF_0, H_1 + H_2 + G_1 + G_2)$ where each $H_i$ has type $(1,0)$
and each $G_j$ has type $(0,1)$;
\item[(b-iii)]
$(\BF_\beta, \Delta_\lambda + \Delta_\infty + F_1 + F_2)$ where
$\Delta_\lambda$ and $\Delta_\infty$ are two sections of $\BF_\beta/\PP^1$ satisfying
$\Delta_\lambda^2 = - \Delta_\infty^2 = \beta\ge 2$ and each $F_i$ is a fiber;
\item[(b-iv)]
$(\PP^2, H+C)$ where $H$ is a line and $C$ is a conic;
\item[(b-v)]
$(\BF_0, C_1 + C_2)$ where each $C_i$ has type $(1,1)$;
\item[(b-vi)]
$(\BF_2, \Delta_0 + \Delta_\lambda)$
or $(\BF_2, \Delta_0 + \Delta_\lambda + \Delta_\infty)$
where $\Delta_0$, $\Delta_\lambda$ and $\Delta_\infty$ are sections of $\BF^2/\PP^1$ satisfying $\Delta_0^2 = \Delta_\lambda^2 = -\Delta_\infty^2 = 2$;
\item[(b-vii)]
$(\BF_\beta, F + \Delta_\infty + C_3)$ where $F$ is a fiber and $\Delta_\infty$ and $C_3$ are two sections of $\BF_\beta/\PP^1$ satisfying $-\Delta_\infty^2 = C_3^2 - 2
= \beta \ge 2$;
\item[(b-viii)]
$(\BF_0, H + G + C)$ where $H, G$ and $C$ has types $(1,0)$, $(0,1)$ and $(1,1)$, respectively;
\item[(b-ix)]
$(\PP^2, E)$ where $E$ is a nodal rational curve with one node;
\item[(b-x)]
$(\BF_0, E)$ where $E$ is a nodal rational curve with one node;
\item[(b-xi)]
$(\BF_2, E)$ or $(\BF_2, E + \Delta_\infty)$, where $E$ is a nodal rational curve with one node and $\Delta_\infty$ is the section of $\BF_2/\PP^1$ with $\Delta_\infty^2 = -2$;
\item[(b-xii)]
$(\BF_0, C_1 + C_2)$ where $C_1$ has type $(1,2)$ and $C_2$ has type $(1,0)$;
\item[(b-xiii)]
$(\BF_\beta, C + \Delta_\infty)$ where $C$ and $\Delta_\infty$ are two sections
of $\BF_\beta/\PP^1$ satisfying $-\Delta_\infty^2 = C^2 - 4 = \beta \ge 2$.
\end{enumerate}
\end{thm}

We use the notations $\mathrm{II}_{\mathrm{a}\text{-}\bullet}$
and $\mathrm{II}_{\mathrm{b}\text{-}\bullet}$ to refer such $(\overline{X}, \overline{D})$.
For the last type $\mathrm{II}_{\mathrm{b}\text{-}\mathrm{xiii}}$,
we may contract $\Delta_\infty$ to obtain a log del Pezzo surface. Thus, we replace/expand this type by/to the following:
\begin{enumerate}
\item[(b-xiii)]
$\overline{X}$ is a log del Pezzo surface of Picard rank $1$, i.e., $\overline{X}$ is a projective surface with log terminal singularities, ample anti-canonical divisor $-K_{\overline{X}}$ and $\rank_\BZ\Pic(\overline{X}) = 1$,
$\overline{D}\sim -K_{\overline{X}}$ is a rational curve with one node $\overline{p}$ and $\overline{X}$ is singular at $\overline{p}$ and smooth outside of $\overline{p}$.
\end{enumerate}
Log del Pezzo surfaces of Picard rank $1$ have been extensively studied (cf. \cite{K-M}). In our case, log del Pezzo surfaces of Picard rank $1$ with a unique singularity were classified by H. Kojima \cite{K}. Although we do not need it here, one can use Kojima's classification to further divide $\mathrm{II}_{\mathrm{b}\text{-}\mathrm{xiii}}$ into subclasses.

We call $(\overline{X}, \overline{D})$ an {\em Iitaka model} of $(X,D)$.
Note that Iitaka model for a log K3 is not unique. For example,
let $X$ be the blowup of $\PP^2$ at two distinct points and let
$D = C_1 + C_2$, where $C_1\sim 2H$ and $C_2\sim H - E_1 - E_2$ with
$H$ the pullback of the hyperplane divisor and $E_i$ the exceptional divisors
of $X\to \PP^2$. We may let $f: X\to \overline{X}\cong \PP^2$
be the blowdown of $E_1$ and $E_2$,
which results in Iitaka model $\text{II}_\text{b-iv}$. Or we may let $f: X\to \overline{X} \cong \BF_0$ be the blowdown of $C_2$, which results in Iitaka model
$\text{II}_\text{b-x}$. Indeed, although we do not need this fact, it is easy to show that there exists a genuine log K3 whose Iitaka model
can be any type in $\text{II}_\text{b}$. 

Also note that although we have $K_X + D = f^* (K_{\overline{X}} + \overline{D})$,
$(\overline{X}, \overline{D})$ is not necessarily a log K3. That is,
$(\overline{X}, \overline{D})$ might be irregular:
\begin{equation}\label{LOGK3RATE102}
h^0(\Omega_{\overline{X}}(\log \overline{D})) = \dim_\BQ \ker(\oplus \BQ \overline{D}_i
\to H^2(\overline{X}, \BQ)) > 0
\end{equation}
for Iitaka types $\text{II}_\text{b-i}$-$\text{II}_\text{b-viii}$, where $\overline{D}_i$ are the irreducible components of $\overline{D}$.

We can reformulate our theorems using the language of Iitaka model:
there are infinitely many $\A^1$ curves in $X\backslash D$ if and only if $(X,D)$ has a log K3 Iitaka model, i.e.,
$\mathrm{II}_{\mathrm{b}\text{-}\mathrm{ix}}$-$\mathrm{II}_{\mathrm{b}\text{-}\mathrm{xiii}}$.

\begin{thm}\label{LOGK3RATTHMLOGISOMIITAKA}
For every genuine log K3 surface $(X,D)$ of type $\mathrm{II}$,
there exists a log isomorphism
$\xymatrix{(X, D) \ar@{-->}[r]^-{\sim} & (\widehat{X}, \widehat{D})}$ followed by a birational morphism $f: \widehat{X} \to\overline{X}$ with $\overline{D} = f_* \widehat{D}$ such that $(\widehat{X}, \widehat{D})$ is one of C0-C4 in Theorem \ref{LOGK3RATTHMLOGISOMK3} and
\begin{itemize}
\item
if $(\widehat{X}, \widehat{D})$ is C0,
$(\overline{X}, \overline{D})$ is one of $\mathrm{II}_{\mathrm{a}\text{-}\bullet}$;
\item
if $(\widehat{X}, \widehat{D})$ is C1,
$(\overline{X}, \overline{D})$ is one of $\mathrm{II}_{\mathrm{b}\text{-}\mathrm{ix}}$-$\mathrm{II}_{\mathrm{b}\text{-}\mathrm{xi}}$;
\item
if $(\widehat{X}, \widehat{D})$ is C2, $(\overline{X}, \overline{D})$ is one of $\mathrm{II}_{\mathrm{b}\text{-}\mathrm{ix}}$-$\mathrm{II}_{\mathrm{b}\text{-}\mathrm{xiii}}$;
\item
if $(\widehat{X}, \widehat{D})$ is C3,
$(\overline{X}, \overline{D})$ is $\mathrm{II}_{\mathrm{b}\text{-}\mathrm{xii}}$;
\item
if $(\widehat{X}, \widehat{D})$ is C4,
$(\overline{X}, \overline{D})$ is $\mathrm{II}_{\mathrm{b}\text{-}\mathrm{iv}}$-$\mathrm{II}_{\mathrm{b}\text{-}\mathrm{vi}}$.
\end{itemize}
\end{thm}

On the other hand, we can prove

\begin{thm}\label{LOGK3RATTHMNOA1}
For each $(\overline{X}, \overline{D})$ among $\mathrm{II}_{\mathrm{b}\text{-}\mathrm{i}}$-$\mathrm{II}_{\mathrm{b}\text{-}\mathrm{viii}}$,
there exists a genuine log K3 surface $(X,D)$ and a birational morphism $g: X\to \overline{X}$ with
$\overline{D} = g_* D$ such that there are at most finitely many $\A^1$ curves in $X\backslash D$.
\end{thm}

The proof of Theorem \ref{LOGK3RATTHMNOA1} gives many examples of genuine log K3 surfaces without infinitely many $\A^1$ curves.

\subsection{Proof of Theorem \ref{LOGK3RATTHMLOGISOMIITAKA}}

If $(\widehat{X}, \widehat{D})$ is of type C0, C1 or C4, 
we simply let $f$ be the blowdown of $X$ to a minimal rational surface $\overline{X}$.

Suppose that $(\widehat{X}, \widehat{D})$ is of type C3. Let $\pi: \widehat{X}\to \PP^1$ be the fiberation given by $|\widehat{D}_2|$. There exists a sequence of blowdowns of $(-1)$-curves contained in the fibers of $\pi$:
\begin{equation}\label{LOGK3RATE113}
\widehat{X} = \widehat{X}_n \xrightarrow{f_n} \widehat{X}_{n-1} \xrightarrow{f_{n-1}}
... \xrightarrow{f_2} \widehat{X}_1 \xrightarrow{f_1} \widehat{X}_0 = \overline{X}
\end{equation}
such that $\pi$ factors through $f = f_1\circ f_2 \circ ... \circ f_n$ and $\overline{X}$ is a rational ruled surface with $\overline{\pi} = \pi\circ f^{-1}: \overline{X}\to \PP^1$. Since $\overline{D} = \overline{D}_1 + \overline{D}_2$ with $\overline{D}_2^2 = 0$, we see that $\overline{X}$ must be either $\BF_0 = \PP^1\times \PP^1$ or $\BF_1$. If it is the former, we are done.

Suppose that $\overline{X} \cong \BF_1$. If $\widehat{X} \not\cong \BF_1$, then
$n\ge 1$ in \eqref{LOGK3RATE113} and let $E_1$ be the exceptional curve of
$f_1$; clearly, there exists another $(-1)$-curve $E_2$ such that
$E_1 + E_2$ is a fiber of $\overline{\pi}\circ f_1: \widehat{X}_1\to \PP^1$ and
the blowdown $f_1': \widehat{X}_1\to \overline{X}'$ of $E_2$ results in
$\overline{X}'\cong \BF_0$. Replacing $(\overline{X}, \overline{D})$ by $(\overline{X}', \overline{D}')$, we are done. If $\widehat{X} \cong \BF_1$, then 
a pivot operation at $\widehat{D}_2$ gives a log isomorphism
$g: (\widehat{X}, \widehat{D})\dashrightarrow (\widehat{X}', \widehat{D}')$
with $\widehat{X}' \cong \BF_0$. Replacing $(\widehat{X}, \widehat{D})$ by $(\widehat{X}', \widehat{D}')$, we are done.

It remains to treat $(\widehat{X}, \widehat{D})$ of type C2. To simplify our notations, we let $(X,D) = (\widehat{X}, \widehat{D})$. We need two lemmas.

\begin{lem}\label{LOGK3RATLEM001}
Let $(X,D)$ be a log surface with $X$ a smooth projective rational surface, $D = D_1 + D_2 + ... + D_n$ a circular boundary
and $K_X + D = 0$.
If there are $n-1$ components $D_2, D_3,..., D_n$ of $D$ such that the intersection matrix of $\{D_2, D_3,...,D_n\}$ is negative definite, then
$(X,D)$ is a genuine log K3.
\end{lem}

\begin{proof}
If $(X,D)$ is not a genuine log K3, then $D_1 = a_2 D_2 + a_3 D_3 + ... + a_n D_n$
for some $a_i \in \BQ$. At least one of $a_i$'s is positive since $D_1, D_2, ..., D_n$ are effective. It follows that
\begin{equation}\label{LOGK3RATE135}
\begin{aligned}
0 &> \left(\sum_{a_i>0} a_i D_i\right)^2 + \left(\sum_{a_i>0} a_i D_i\right) \left(\sum_{a_j < 0} a_j D_j\right)\\
&= D_1 \left(\sum_{a_i>0} a_i D_i\right) \ge 0
\end{aligned}
\end{equation}
since the intersection matrix of $\{D_2, D_3,...,D_n\}$ is negative definite.
Contradiction. Therefore, $(X,D)$ is a genuine log K3.
\end{proof}

\begin{lem}\label{LOGK3RATLEM003}
Let $(X,D)$ be a genuine log K3 surface with circular boundary $D = D_1 + D_2 + ... + D_n$ satisfying that
$D_1^2 > 0$ and $D_i^2 \le -2$ for $i\ne 1$.
If $\rank_\BZ \Pic(X) > n$, then
there exists a nontrivial birational morphism $f: X\to \overline{X}$
with $\overline{D} = f_* D$ such that
$(\overline{X},\overline{D})$ is a genuine log K3 surface of type C1, C3 or with the property that
\begin{equation}\label{LOGK3RATE1B2}
\begin{aligned}
&\text{there exists an irreducible component } \overline{D}_i\subset \overline{D}\\
&\quad\text{such that }\overline{D} - \overline{D}_i \text{ has negative definite intersection matrix}.
\end{aligned}
\end{equation}
\end{lem}

\begin{proof}
Note that the number $\rank_\BZ \Pic(X) - \mu(D)$ remains the same after 
a canonical blowdown and decrease by one after a contraction of a $(-1)$-curve not contained in $D$. Suppose that $D_i$'s satisfy \eqref{LOGK3RATE103}.

If there is a $(-1)$-curve $E$ meets $D_1$, we may simply blow down $E$ and 
the resulting $(\overline{X},\overline{D})$ obviously satisfies \eqref{LOGK3RATE1B2}. Let us assume that 
\begin{equation}\label{LOGK3RATE500}
D_1 E = 0 \text{ for all $(-1)$-curves } E.
\end{equation}

Clearly, $X\not\cong \PP^2$. So there exists a fiberation $g: X\to \PP^1$ whose general fibers are $\PP^1$. Since $D_1^2 > 0$, $\pi_* D_1 \ne 0$.
If $g$ has a reducible fiber $F_r$ meeting $D$ properly, then $F_r$ has a component $E$ such that $D_1 E > 0$; since
$E^2 < 0$ and $KE < 0$, $E$ must be a $(-1)$-curve. This is impossible by \eqref{LOGK3RATE500}. So
\begin{equation}\label{LOGK3RATE501}
F_r\cong \PP^1 \text{ for all fibers } F_r \text{ of } g \text{ satisfying }\dim (F_r\cap D) = 0.
\end{equation}
Obviously,
$D_1$ is either a section or a multi-section of degree $2$ of $g$.
If it is the latter, then $D_2 + D_3 + ... + D_n$ is contained in a fiber of $g$.
Suppose that $g$ factors through a ruled surface $\overline{X}$. Then
$g_* D = \overline{D}$ is either a nodal rational curve or
has two components $\overline{D} = \overline{D}_1 + \overline{D}_2$ with
$\overline{D}_2^2 = 0$. Namely, $(\overline{X},\overline{D})$ is a genuine log K3 surface of type C1 or C3.

Let us assume that $D_1$ is a section of $g$. Then there is another component $D_i$ that is a section of $g$ and the rest $D - D_1 - D_i$ are contained in the two fibers
$F_p$ and $F_q$ of $g$ over $p\ne q\in \PP^1$. By \eqref{LOGK3RATE501},
$F_r\cong \PP^1$ for all $F_r\ne F_p, F_q$. Therefore, we have
\begin{equation}\label{LOGK3RATE129}
\begin{aligned}
&\quad \rank_\BZ \Pic(X) - \mu(D) = \mu(F_p) + \mu(F_q) - \mu(D)\\ 
&= (\mu(F_p) - \mu(F_p\cap D) - 1) + (\mu(F_q) - \mu(F_q\cap D) - 1) > 0.
\end{aligned}
\end{equation}
Consequently, either $\mu(F_p) \ge \mu(F_p\cap D) + 2$
or $\mu(F_q) \ge \mu(F_q\cap D) + 2$. WLOG, suppose that
\begin{equation}\label{LOGK3RATE128}
\mu(F_p) \ge \mu(F_p\cap D) + 2.
\end{equation}
It follows that $F_p$ contains
\begin{itemize}
	\item either one $(-1)$-curve $E_1$ and one $(-2)$-curve $E_2$ with the properties
	$E_2\cap D = \emptyset$ and $E_1 E_2 = 1$
	\item or two disjoint $(-1)$-curves $E_1$ and $E_2$.
\end{itemize}

Let $\phi: X\to \overline{X}$ be the contraction of $E_1$ followed by a sequence of blowdowns of $(-1)$-curves contained in $F_p\cap D$
such that $\overline{F}_p\cap \overline{D}$ does not contain any $(-1)$-curves
for $\overline{D} = \phi_* D$ and $\overline{F}_p = \phi_* F_p$. That is, $\phi$ is a birational morphism with the commutative diagram
\begin{equation}\label{LOGK3RATE130}
\xymatrix{
	X \ar[r]^-\phi\ar[d]_-g & \overline{X} \ar[dl]\\
	\PP^1
}
\end{equation}
such that $\overline{X}$ smooth, $\overline{F}_p\cap \overline{D}$ does not contain any $(-1)$-curves and the exceptional locus $E_\phi$ of $\phi$ satisfies $E_1\subset E_\phi \subset E_1\cup D$.

Suppose that $E_2^2 = -2$. WLOG, suppose that $E_1 D_j = 1$ for some $j > i$.
Since $(\phi_* E_2)^2 \le 0$, $E_\phi$ consists of at most
two components. If $E_\phi = E_1$, then all components of $\overline{D}$ other than $\overline{D}_1 = \phi_* D_1$ still have self-intersections $\le -2$ and hence
$(\overline{X}, \overline{D})$ satisfies \eqref{LOGK3RATE1B2}. If $E_\phi = E_1 + D_j$ has two components,
then $(\phi_* E_2)^2 = 0$ and we must have $\phi(E_2) = \overline{F}_p$. That is, $\phi$ contracts all components $F_p\cap D$. So $\overline{D}_1\cap \overline{D}_i\cap \overline{F}_p \ne \emptyset$
for $\overline{D}_i = \phi_* D_i$.
And since $E_\phi\cap D$ has one component, $\overline{D}_i^2 = D_i^2 + 1 \le -1$. On the other hand, all components of $\overline{D}$ other than $\overline{D}_1$ and $\overline{D}_i$ still have self-intersections $\le -2$. 
That is, $\overline{D}_i$ is a circular boundary of type
$(\lambda_1, \lambda_2, ..., \lambda_{i-1}, \overline{\lambda}_i)$
with $\lambda_k = \overline{D}_k^2 = D_k^2 \le -2$ for $2\le k < i$ and
$\overline{\lambda}_i = \overline{D}_i^2 \le -1$.
Therefore, the components of $\overline{D} - \overline{D}_1$ still have negative definite intersection matrix. So $(\overline{X}, \overline{D})$ satisfies  \eqref{LOGK3RATE1B2} again.

Suppose that $E_2^2 = -1$. WLOG, suppose that
$E_1 D_{j_1} = 1$ and $E_2 D_{j_2} = 1$ for some $j_1 \ge j_2 > i$.
If $E_\phi \cap E_2 = \emptyset$, then $\overline{D}_i^2 = D_i^2 \le -2$ and all components of $\overline{D}$ other than $\overline{D}_1$ still have self-intersections $\le -2$ and hence
$(\overline{X}, \overline{D})$ satisfies \eqref{LOGK3RATE1B2}.
If $E_\phi \cap E_2 \ne \emptyset$, then $\phi_* E_2 = \overline{F}_p$ and $\phi$ contracts all components $F_p\cap D$. So $\overline{D}_1\cap \overline{D}_i\cap \overline{F}_p \ne \emptyset$.
And since $j_1\ge j_2>i$, $\overline{D}_i^2 = D_i^2 + 1 \le -1$. On the other hand, all components of $\overline{D}$ other than $\overline{D}_1$ and $\overline{D}_i$ still have self-intersections $\le -2$. 
That is, $\overline{D}_i$ is a circular boundary of type
$(\lambda_1, \lambda_2, ..., \lambda_{i-1}, \overline{\lambda}_i)$
with $\lambda_k = \overline{D}_k^2 = D_k^2 \le -2$ for $2\le k < i$ and
$\overline{\lambda}_i = \overline{D}_i^2 \le -1$.
Therefore, the components of $\overline{D} - \overline{D}_1$ still have negative definite intersection matrix. So $(\overline{X}, \overline{D})$ satisfies \eqref{LOGK3RATE1B2} again.
\end{proof}

Now we can complete the proof of Theorem \ref{LOGK3RATTHMLOGISOMIITAKA}.

Suppose that $D = D_1 + D_2 + ... + D_n$ is a circular boundary of type $(\lambda_1, \lambda_2, ..., \lambda_n)$ with $D_i^2 = \lambda_i \le -2$ for all $i\ne 1$.
We argue by induction on $\rank_\BZ \Pic(X)$.

By Lemma \ref{LOGK3RATLEMD12}, we can reduce it to the case
$D_1^2 > 0$. If $\rank_\BZ \Pic(X) > n$, we may apply Lemma
\ref{LOGK3RATLEM003} to reduce $\rank_\BZ \Pic(X)$ by $1$.
If $\rank_\BZ \Pic(X) = n$, we may contract
the rod $D_2 + D_3 + ... + D_n$ to obtain a log del Pezzo surface $\overline{X}$ of Picard rank $1$. Indeed, we can be more precise about the singularity $\overline{p}$:
$\overline{X}$ has a cyclic quotient singularity at $\overline{p}$
given by $\BC^2/(\exp(2\pi i/a), \exp(2\pi i b/a))$, where
\begin{equation}\label{LOGK3RATE131}
\frac{a}{b} = -\lambda_2 + \frac{1}{\lambda_3 + 
\displaystyle{\frac{1}{\lambda_4 + ...}}}.
\end{equation}

\subsection{Proof of Theorem \ref{LOGK3RATTHMNOA1}}
	
It suffices to blow up each Iitaka model $\overline{X}$ at some smooth points of $\overline{D}$, called {\em half point attachments} by Iitaka, such that the resulting $(X, D)$ is a log K3, i.e., $h^0(\Omega_X(\log D)) = 0$,
and $(X,D)$ fails \eqref{LOGK3RATEB2} (see Figure \ref{LOGK3RATFIGB2B4000}). If $h^0(\Omega_{\overline{X}}(\log \overline{D})) = r$, we need to blow up $r$ points, i.e., attach $r$ half points.

For $\text{II}_\text{b-i}$, it suffices to blow up $\overline{X}\cong \PP^2$ at
one point on $H_1$ and one point on $H_2$.
Then $(X,D)$ fails \eqref{LOGK3RATEB2} since all three components of $D$ have self-intersections $\ge 0$.

For $\text{II}_\text{b-ii}$, it suffices to blow up $\overline{X}\cong \BF_0$ at
one point on $H_1$ and one point on $G_1$.
Then $(X,D)$ fails \eqref{LOGK3RATEB2} since all four components of $D$ have self-intersections $\ge -1$.

For $\text{II}_{\text{b-iii}}$,
it suffices to blow up $\overline{X}\cong \BF_\beta$ at one point on $F_1$ and one point on $\Delta_\lambda$.
After a sequence of pivot operations at $F_2$, we arrive at a log isomorphism
$(X,D)\dashrightarrow (\widehat{X}, \widehat{D})$, where $\widehat{D}$ has four components with self-intersections $-1, -1, 0, 0$, respectively, and \eqref{LOGK3RATEB2} fails.

For $\text{II}_\text{b-iv}$, it suffices to blow up $\overline{X}\cong \PP^2$ at
one point on $C$. Then $(X,D)$ fails \eqref{LOGK3RATEB2} since both components of $D$ have self-intersections $\ge 1$ (see Figure \ref{LOGK3RATFIGB2B3000}).

For $\text{II}_\text{b-v}$, it suffices to blow up $\overline{X}\cong \BF_0$ at
one point on $C_1$. Then $(X,D)$ fails \eqref{LOGK3RATEB2} since both components of $D$ have self-intersections $\ge 1$.

For $\text{II}_\text{b-vi}$, it suffices to blow up $\overline{X}\cong \BF_2$ at
one point on $\Delta_0$. Then $(X,D)$ fails \eqref{LOGK3RATEB2} since both components of $D$ have self-intersections $\ge 1$.

For $\text{II}_\text{b-vii}$, it suffices to blow up $\overline{X}\cong \BF_\beta$ at
one point on $C_3$. After a sequence of pivot operations at $F$, we arrive at a log isomorphism
$(X,D)\dashrightarrow (\widehat{X}, \widehat{D})$, where $\widehat{D}$ has three components with self-intersections $0, 0, 1$, respectively, and \eqref{LOGK3RATEB2} fails.

For $\text{II}_\text{b-viii}$, it suffices to blow up $\overline{X}\cong \BF_0$ at one point on $C$.
Then $(X,D)$ fails \eqref{LOGK3RATEB2} since all three components of $D$ have self-intersections $\ge 0$.

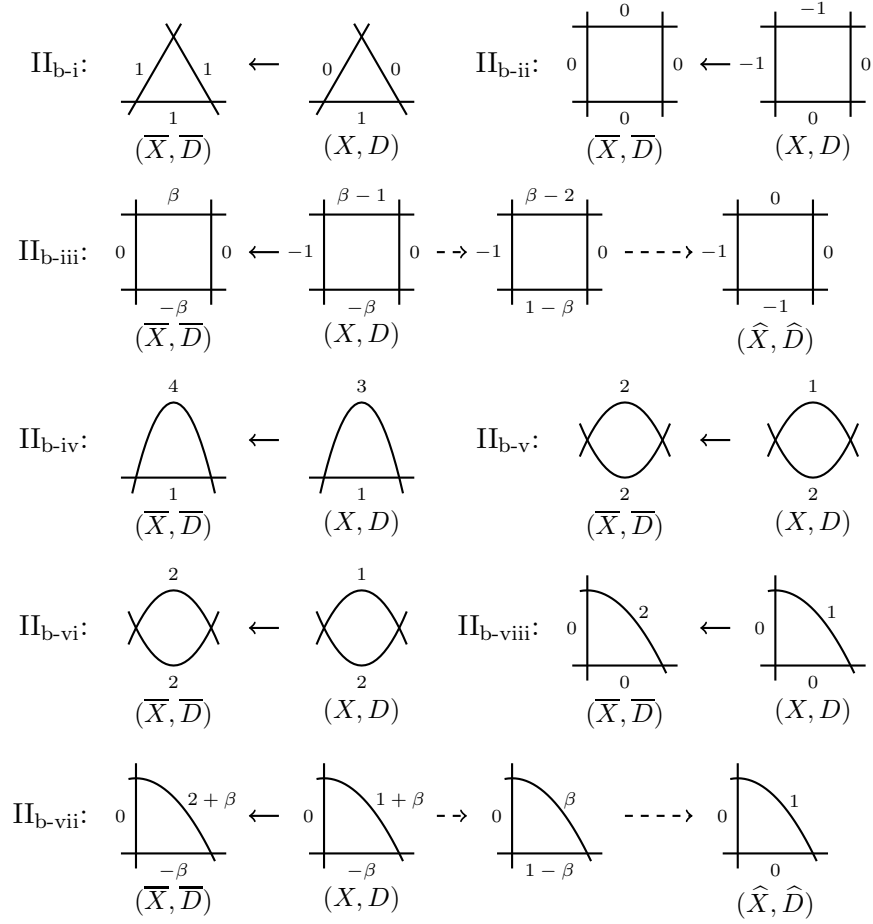
\begin{figure}
\begin{tikzpicture}[scale=1]
\node[left] at (-0.5,0.5) {{$\text{II}_{\text{b-i}}$:}};

\draw[thick] (-0.2,0) -- (1.2,0);
\node[below] at (0.5,0) {{\tiny $1$}};
\node[below] at (0.5,-0.3) {{\small $(\overline{X},\overline{D})$}};
\draw[thick] (-0.1,-.1732) -- (0.6,1.0392);
\node[above,left] at (0.25,0.433) {{\tiny $1$}};
\draw[thick] (1.1,-.1732) -- (0.4, 1.0392);
\node[above,right] at (0.75,0.433) {{\tiny $1$}};

\draw[->, thick] (1.9,0.5) -- (1.5,0.5);

\draw[thick] (2.3,0) -- (3.7,0);
\node[below] at (3,0) {{\tiny $1$}};
\node[below] at (3,-0.3) {{\small $(X,D)$}};
\draw[thick] (2.4,-.1732) -- (3.1,1.0392);
\node[above,left] at (2.75,0.433) {{\tiny $0$}};
\draw[thick] (3.6,-.1732) -- (2.9, 1.0392);
\node[above,right] at (3.25,0.433) {{\tiny $0$}};

\node[left] at (5.5,0.5) {{$\text{II}_{\text{b-ii}}$:}};

\draw[thick] (5.8,0) -- (7.2,0);
\node[below] at (6.5,0) {{\tiny $0$}};
\node[below] at (6.5,-0.3) {{\small $(\overline{X}, \overline{D})$}};
\draw[thick] (6,-0.2) -- (6,1.2);
\node[left] at (6,0.5) {{\tiny $0$}};
\draw[thick] (5.8,1) -- (7.2,1);
\node[above] at (6.5,1) {{\tiny $0$}};
\draw[thick] (7,-0.2) -- (7,1.2);
\node[right] at (7,0.5) {{\tiny $0$}};

\draw[->, thick] (7.9,0.5) -- (7.5,0.5);

\draw[thick] (8.3,0) -- (9.7,0);
\node[below] at (9,0) {{\tiny $0$}};
\node[below] at (9,-0.3) {{\small $(X, D)$}};
\draw[thick] (8.5,-0.2) -- (8.5,1.2);
\node[left] at (8.5,0.5) {{\tiny $-1$}};
\draw[thick] (8.3,1) -- (9.7,1);
\node[above] at (9,1) {{\tiny $-1$}};
\draw[thick] (9.5,-0.2) -- (9.5,1.2);
\node[right] at (9.5,0.5) {{\tiny $0$}};

\node[left] at (-0.5,-2) {{$\text{II}_{\text{b-iii}}$:}};

\draw[thick] (-0.2,-2.5) -- (1.2,-2.5);
\node[below] at (.5,-2.5) {{\tiny $-\beta$}};
\node[below] at (.5,-2.8) {{\small $(\overline{X}, \overline{D})$}};
\draw[thick] (0,-2.7) -- (0,-1.3);
\node[left] at (0,-2) {{\tiny $0$}};
\draw[thick] (-0.2,-1.5) -- (1.2,-1.5);
\node[above] at (.5,-1.5) {{\tiny $\beta$}};
\draw[thick] (1,-2.7) -- (1,-1.3);
\node[right] at (1,-2) {{\tiny $0$}};

\draw[->, thick] (1.9,-2) -- (1.5,-2);

\draw[thick] (2.3,-2.5) -- (3.7,-2.5);
\node[below] at (3,-2.5) {{\tiny $-\beta$}};
\node[below] at (3,-2.8) {{\small $(X, D)$}};
\draw[thick] (2.5,-2.7) -- (2.5,-1.3);
\node[left] at (2.5,-2) {{\tiny $-1$}};
\draw[thick] (2.3,-1.5) -- (3.7,-1.5);
\node[above] at (3,-1.5) {{\tiny $\beta-1$}};
\draw[thick] (3.5,-2.7) -- (3.5,-1.3);
\node[right] at (3.5,-2) {{\tiny $0$}};

\draw[->, dashed, thick] (4,-2) -- (4.4,-2);

\draw[thick] (4.8,-2.5) -- (6.2,-2.5);
\node[below] at (5.5,-2.5) {{\tiny $1-\beta$}};
\draw[thick] (5,-2.7) -- (5,-1.3);
\node[left] at (5,-2) {{\tiny $-1$}};
\draw[thick] (4.8,-1.5) -- (6.2,-1.5);
\node[above] at (5.5,-1.5) {{\tiny $\beta-2$}};
\draw[thick] (6,-2.7) -- (6,-1.3);
\node[right] at (6,-2) {{\tiny $0$}};

\draw[->, dashed, thick] (6.5,-2) -- (7.4,-2);

\draw[thick] (7.8,-2.5) -- (9.2,-2.5);
\node[below] at (8.5,-2.5) {{\tiny $-1$}};
\node[below] at (8.5,-2.8) {{\small $(\widehat{X}, \widehat{D})$}};
\draw[thick] (8,-2.7) -- (8,-1.3);
\node[left] at (8,-2) {{\tiny $-1$}};
\draw[thick] (7.8,-1.5) -- (9.2,-1.5);
\node[above] at (8.5,-1.5) {{\tiny $0$}};
\draw[thick] (9,-2.7) -- (9,-1.3);
\node[right] at (9,-2) {{\tiny $0$}};

\node[left] at (-0.5,-4.5) {{$\text{II}_{\text{b-iv}}$:}};

\draw[thick] (-0.2, -5) -- (1.2,-5);
\node[below] at (0.5,-5) {{\tiny $1$}};
\node[below] at (0.5,-5.3) {{\small $(\overline{X}, \overline{D})$}};
\draw[thick, domain=-0.05:1.05] plot (\x, {4*(\x - \x*\x) - 5});
\node[above] at (0.5,-4) {{\tiny $4$}};

\draw[->, thick] (1.9,-4.5) -- (1.5,-4.5);

\draw[thick] (2.3, -5) -- (3.7,-5);
\node[below] at (3,-5) {{\tiny $1$}};
\node[below] at (3,-5.3) {{\small $(X, D)$}};
\draw[thick, domain=-0.05:1.05] plot (\x + 2.5, {4*(\x - \x*\x) - 5});
\node[above] at (3,-4) {{\tiny $3$}};

\node[left] at (5.5,-4.5) {{$\text{II}_{\text{b-v}}$:}};

\draw[thick, domain=-0.1:1.1] plot (\x + 6, {2*(\x - \x*\x) - 4.5});
\draw[thick, domain=-0.1:1.1] plot (\x + 6, {2*(-\x + \x*\x) - 4.5});
\node[below] at (6.5,-5) {{\tiny $2$}};
\node[below] at (6.5,-5.3) {{\small $(\overline{X}, \overline{D})$}};
\node[above] at (6.5,-4) {{\tiny $2$}};

\draw[->, thick] (7.9,-4.5) -- (7.5,-4.5);

\draw[thick, domain=-0.1:1.1] plot (\x + 8.5, {2*(\x - \x*\x) - 4.5});
\draw[thick, domain=-0.1:1.1] plot (\x + 8.5, {2*(-\x + \x*\x) - 4.5});
\node[below] at (9,-5) {{\tiny $2$}};
\node[below] at (9,-5.3) {{\small $(X, D)$}};
\node[above] at (9,-4) {{\tiny $1$}};

\node[left] at (-0.5,-7) {{$\text{II}_{\text{b-vi}}$:}};

\draw[thick, domain=-0.1:1.1] plot (\x, {2*(\x - \x*\x) - 7});
\draw[thick, domain=-0.1:1.1] plot (\x, {2*(-\x + \x*\x) - 7});
\node[below] at (.5,-7.5) {{\tiny $2$}};
\node[below] at (.5,-7.8) {{\small $(\overline{X}, \overline{D})$}};
\node[above] at (.5,-6.5) {{\tiny $2$}};

\draw[->, thick] (1.9,-7) -- (1.5,-7);

\draw[thick, domain=-0.1:1.1] plot (\x+2.5, {2*(\x - \x*\x) - 7});
\draw[thick, domain=-0.1:1.1] plot (\x+2.5, {2*(-\x + \x*\x) - 7});
\node[below] at (3,-7.5) {{\tiny $2$}};
\node[below] at (3,-7.8) {{\small $(X, D)$}};
\node[above] at (3,-6.5) {{\tiny $1$}};

\node[left] at (5.5,-7) {{$\text{II}_{\text{b-viii}}$:}};

\draw[thick] (5.8, -7.5) -- (7.2,-7.5);
\node[below] at (6.5,-7.5) {{\tiny $0$}};
\node[below] at (6.5,-7.8) {{\small $(\overline{X}, \overline{D})$}};
\draw[thick] (6,-7.7) -- (6,-6.3);
\node[left] at (6,-7) {{\tiny $0$}};
\draw[thick, domain=-0.1:1.05] plot (\x+6, {(1 - \x*\x) - 7.5});
\node[above] at (6.75,-7) {{\tiny $2$}};

\draw[->, thick] (7.9,-7) -- (7.5,-7);

\draw[thick] (8.3, -7.5) -- (9.7,-7.5);
\node[below] at (9,-7.5) {{\tiny $0$}};
\node[below] at (9,-7.8) {{\small $(X, D)$}};
\draw[thick] (8.5,-7.7) -- (8.5,-6.3);
\node[left] at (8.5,-7) {{\tiny $0$}};
\draw[thick, domain=-0.1:1.05] plot (\x+8.5, {(1 - \x*\x) - 7.5});
\node[above] at (9.25,-7) {{\tiny $1$}};

\node[left] at (-0.5,-9.5) {{$\text{II}_{\text{b-vii}}$:}};

\draw[thick] (-0.2, -10) -- (1.2,-10);
\node[below] at (0.5,-10) {{\tiny $-\beta$}};
\node[below] at (0.5,-10.3) {{\small $(\overline{X}, \overline{D})$}};
\draw[thick] (0,-10.2) -- (0,-8.8);
\node[left] at (0,-9.5) {{\tiny $0$}};
\draw[thick, domain=-0.1:1.05] plot (\x, {(1 - \x*\x) - 10});
\node[right] at (0.55,-9.3) {{\tiny $2+\beta$}};

\draw[->, thick] (1.9,-9.5) -- (1.5,-9.5);

\draw[thick] (2.3, -10) -- (3.7,-10);
\node[below] at (3,-10) {{\tiny $-\beta$}};
\node[below] at (3,-10.3) {{\small $(X, D)$}};
\draw[thick] (2.5,-10.2) -- (2.5,-8.8);
\node[left] at (2.5,-9.5) {{\tiny $0$}};
\draw[thick, domain=-0.1:1.05] plot (\x + 2.5, {(1 - \x*\x) - 10});
\node[right] at (3.05,-9.3) {{\tiny $1+\beta$}};

\draw[->, dashed, thick] (4,-9.5) -- (4.4,-9.5);

\draw[thick] (4.8, -10) -- (6.2,-10);
\node[below] at (5.5,-10) {{\tiny $1-\beta$}};
\draw[thick] (5,-10.2) -- (5,-8.8);
\node[left] at (5,-9.5) {{\tiny $0$}};
\draw[thick, domain=-0.1:1.05] plot (\x + 5, {(1 - \x*\x) - 10});
\node[right] at (5.55,-9.3) {{\tiny $\beta$}};

\draw[->, dashed, thick] (6.5,-9.5) -- (7.4,-9.5);

\draw[thick] (7.8, -10) -- (9.2,-10);
\node[below] at (8.5,-10) {{\tiny $0$}};
\node[below] at (8.5,-10.3) {{\small $(\widehat{X}, \widehat{D})$}};
\draw[thick] (8,-10.2) -- (8,-8.8);
\node[left] at (8,-9.5) {{\tiny $0$}};
\draw[thick, domain=-0.1:1.05] plot (\x + 8, {(1 - \x*\x) - 10});
\node[right] at (8.55,-9.3) {{\tiny $1$}};

\end{tikzpicture}
\caption{Log K3 surfaces of type $\text{II}_{\text{b-i}}$-$\text{II}_{\text{b-viii}}$ without infinitely many $\A^1$ curves}
\label{LOGK3RATFIGB2B4000}
\end{figure}

\end{document}